\documentclass[12pt,a4paper]{amsart}

\usepackage[T1]{fontenc}
\usepackage[utf8]{inputenc}

\usepackage{amsmath}
\usepackage{amsthm}
\usepackage{amssymb}
\usepackage{mathtools}
\usepackage{enumitem}

\usepackage{mathrsfs}
\usepackage{stmaryrd}

\DeclareFontFamily{U}{MnSymbolC}{}
\DeclareSymbolFont{MnSyC}{U}{MnSymbolC}{m}{n}
\DeclareMathSymbol{\diamondvert}{\mathbin}{MnSyC}{"79}
\DeclareMathSymbol{\Diamond}{\mathbin}{MnSyC}{"6E}
\DeclareFontShape{U}{MnSymbolC}{m}{n}{
    <-6>  MnSymbolC5
   <6-7>  MnSymbolC6
   <7-8>  MnSymbolC7
   <8-9>  MnSymbolC8
   <9-10> MnSymbolC9
  <10-12> MnSymbolC10
  <12->   MnSymbolC12}{}

\usepackage{nicefrac}

\usepackage[numbers,sort&compress,comma]{natbib}

\usepackage{xspace}
\usepackage{xcolor}

\usepackage[margin=3.25cm]{geometry}

\usepackage[hypertexnames=false]{hyperref}
\hypersetup{
    colorlinks,
    linkcolor={red!90!black},
    citecolor={green!70!black},
    urlcolor={blue!80!black}
}

\usepackage{tikz}
\usetikzlibrary{calc}
\usetikzlibrary{shapes.geometric}
\usetikzlibrary{positioning}

\usepackage[format=hang,font={footnotesize}]{caption}

\newcommand\Betw{B} 

\newcommand{\Bleqs}{B_{\leqslant}} 

\newcommand{\uv}{\underline{v}}
\newcommand{\uP}{\underline{P}}
\newcommand{\uR}{\underline{R}}
\newcommand{\uS}{\underline{S}}
 
\newcommand{\uW}{\underline{W}}

\newcommand{\uUltA}{\underline{\Ult A}}
\newcommand{\uUltN}{\underline{\Ult 2^\omega}}
\newcommand{\ulteps}{\ult^{\epsilon}} 
\newcommand{\ultVeps}{\ultV^{\epsilon}}
\newcommand\QB{Q_{\poss{B_\leqslant}}} 
\newcommand\SB{S_{\suff{B_\leqslant}}} 

\usepackage[all]{xy}

\newcommand{\poss}[1]{\ensuremath{\langle { #1 \rangle}}}
\newcommand{\suff}[1]{\ensuremath{[[ #1 ]]}}
\newcommand{\set}[1]{\ensuremath{\{#1\}}}
\newcommand{\klam}[1]{\ensuremath{\langle #1 \rangle}}

\newcommand{\tor}{\text{ or }}
\newcommand{\tand}{\text{ and }}
\newcommand{\qtand}{\quad\text{and}\quad}
\newcommand{\Va}{\ensuremath{\mathbf{V}}\xspace}
\newcommand{\win}{\boxbar}
\newcommand{\ua}[1]{\ensuremath{\mathop{\uparrow}#1}}

\renewcommand{\phi}{\varphi}

\newcommand{\df}{\defeq}
\newcommand{\tiff}{if and only if\xspace}
\newcommand{\aright}{($\Rarrow$)\xspace}
\newcommand{\aleft}{($\Larrow$)\xspace}

\DeclareMathOperator{\ultV}{\mathscr{V}}
\DeclareMathOperator{\ultW}{\mathscr{W}}
\DeclareMathOperator{\Eq}{\mathbf{V}}
\DeclareMathOperator{\Abtw}{\mathbf{Abtw}}

\newcommand{\Cm}{\ensuremath{\operatorname{{\mathsf{Cm}}}}}
\newcommand{\Cf}{\ensuremath{\operatorname{{\mathsf{Cf}}}}}

 \newcommand{\Var}{\ensuremath{\mathsf{Var}}}
\newcommand{\Fml}{\ensuremath{\mathsf{Fml}}}
\newcommand{\Kpm}{\ensuremath{K^{\#}}\xspace}

\newcommand{\B}{\ensuremath{\mathfrak{B}}}
\newcommand{\wmia}{\ensuremath{\mathbf{wMIA}}\xspace}
\newcommand{\dmia}{\ensuremath{\mathbf{dMIA}}\xspace}

\newcommand{\necu}[1]{\ensuremath{[ #1 ]]}}
\newcommand{\possu}[1]{\ensuremath{\klam{ #1 }\!\rangle}}
\newcommand{\symsum}{\mathrel{\nabla}}
\newcommand{\fd}{f^\partial}

\newcommand{\zero}{\mathbf{0}} 
\newcommand{\one}{\mathbf{1}} 
\newcommand\defeq{\coloneqq} 
\newcommand\Rarrow{\Rightarrow}
\newcommand\Larrow{\Leftarrow}
\newcommand\rarrow{\rightarrow}
\newcommand\Iff{\Leftrightarrow}
\newcommand\iffdef{\;\mathrel{\mathord{:}\mathord{\Leftrightarrow}}\;}
\newcommand\dftt{\mathtt{df}\,}
\DeclareMathOperator{\Ult}{Ult} 
\newcommand{\ult}{\mathscr{U}} 

\newcommand\mfr{\mathfrak}
\newcommand\frA{\mfr{A}}
\newcommand\frF{\mfr{F}}
\newcommand\frM{\mfr{M}}
\newcommand\frN{\mfr{N}}

\newcommand*{\inference}[3][c]{%
   \begingroup
   \def\and{\\}
   \begin{tabular}[#1]{@{\enspace}c@{\enspace}}
   $#2$ \\
   \hline
   $#3$
   \end{tabular}%
   \endgroup
}

\numberwithin{equation}{section}

\newtheorem{theorem}{Theorem}[section]
\newtheorem{lemma}[theorem]{Lemma}
\newtheorem{proposition}[theorem]{Proposition}
\newtheorem{corollary}[theorem]{Corollary}

\newtheoremstyle{mytheoremstyle} 
    {1em plus .2em minus .1em}                    
    {1em plus .2em minus .1em}                    
    {\rmfamily}                   
    {}                           
    {\bfseries}                   
    {.}                          
    {.5em}                       
    {}  

\theoremstyle{mytheoremstyle}

\newtheorem{definition}[theorem]{Definition}

\newtheorem{remark}[theorem]{Remark}

\theoremstyle{definition}

\renewcommand{\Iff}{\Leftrightarrow}

\numberwithin{equation}{section}

\date{}

\title{A mixed logic with binary operators}
\author{Ivo D\"untsch, Rafa\l\ Gruszczy\'nski, Paula Mench\'on}

\address{Ivo D\"untsch\\
Department of Computer Science\\
Brock University\\	
St Ca\-tha\-ri\-nes, Ontario\\
Canada\\
\textsc{Orcid:} 0000-0001-8907-2382}

\email{duentsch@brocku.ca}

\urladdr{https://www.cosc.brocku.ca/~duentsch/}

\address{Rafa\l\ Gruszczy\'nski\\
Department of Logic\\
Nicolaus Copernicus University in Toru\'n\\
Poland\\
\textsc{Orcid:} 0000-0002-3379-0577}
\email{gruszka@umk.pl}

\urladdr{www.umk.pl/\textasciitilde gruszka}

\address{Paula Mench\'on\\
Department of Logic\\
Nicolaus Copernicus University in Toru\'n\\
Poland\\
and 
Universidad Nacional del Centro de la Provincia de Buenos Aires\\ 
Tandil, Buenos Aires\\ 
Argentina\\
\textsc{Orcid:} 0000-0002-9395-107X
}

\email{mpmenchon@nucompa.exa.unicen.edu.ar}

\begin{document}

\begin{abstract}
In previous work \citep{dgm23} we introduced and examined the class of betweenness algebras. In the current paper we study a larger class of algebras with binary operators of possibility and sufficiency, the \textit{weak mixed algebras}. Furthermore, we develop a system of logic with two binary modalities, sound and complete with respect to the class of frames closely related to the aforementioned algebras, and we prove an embedding theorem which solves an open problem from \citep{dgm23}.

\medskip 


    \medskip

\end{abstract}

\maketitle

\vspace{1cm}

\begin{flushright}
    \parbox{0.50\textwidth}{\emph{We dedicate this paper to our esteemed\\ colleague and friend Dimiter Vakarelov.}}
\end{flushright}

\vspace{1cm}

\section{Introduction}

Possibility-sufficiency algebras (PS-algebras), as Boolean algebras expanded with unary operators of possibility and sufficiency, were first studied in \citep{Duntsch-et-al-MMASO}. The latter operator is an algebraic counterpart of the window modality, whose provenance can be traced back to \cite{Humberstone-IW,gpt87,gor88}, and also \cite{vanBenthem-MDL} in the framework of deontic logic. The theory of PS-algebras was refined and developed further in \cite{do_mixalg} and \cite{dot_mixed}, with particular emphasis on two subclasses of PS-algebras: mixed and weak mixed algebras, which force a relationship between the---otherwise independent---operators of possibility and sufficiency.

Building on the aforementioned works, we initiated in \citep{dgm23} a theory of betweenness algebras (b-algebras), which were thought of as an approach to study most general algebraic and logical properties of the ternary relation of betweenness well-known from ordered geometry. Being ternary, betweenness calls for binary operators, and for this reason, our PS-algebras are Boolean algebras with binary operators of possibility and sufficiency. To guarantee the interplay between the two operators, we focused on weak mixed algebras, which we enriched with equalities expressing the basic and well-motivated properties of betweenness. The approach turned out to be fruitful and allowed for the development of both representation and canonical extensions of algebras arising from betweenness relations. For these, the canonical frames of b-algebras are triples of the form $\klam{W,R,S}$, where $S$ and $R$ are ternary relations on $W$ obtained in the standard way from sufficiency and possibility operators, respectively, with $S\subseteq R$. In a certain sense, $R$ and $S$ together simulate the betweenness relation. We showed that this cannot be avoided, in the sense that, in general, the canonical frame of a b-algebra cannot have just one ternary relation which is a betweenness. This was a negative result, as the starting point of our research were precisely frames $\klam{W,B}$, where $B$ is a betweenness, and the properties of the complex algebras of such frames gave rise to the axiomatization presented in the paper. However, the negative finding did not exclude existence of a more general ternary relation that could allow for a representation theorem for b-algebras, and finding such a relation was formulated as problem (1) in \citep{dgm23}. 

With the problem in focus, we concentrate below on a general study of frames $\klam{W,R,S}$ with $S\subseteq R$ (wMIA frames). We develop a logic $\Kpm$ with binary operators of possibility and sufficiency (window modality) that is sound and complete with respect to the class of wMIA frames. To this end, we adapt for our needs the copying method from \citep{vak89,gpt87}, which allows for the transformation of wMIA frames into frames with a single ternary relation. Further, we study algebraic models of $\Kpm$, i.e., PS-algebras determined by an equational theory $\Sigma$. We prove that the variety $\Eq(\Sigma)$ is precisely the variety generated by the class of weak mixed algebras (which does not form a~variety itself). 

The techniques and results mentioned above allow us to prove that every weak mixed algebra can be embedded into a frame of the form $\klam{W,R}$, and in consequence, solve the problem (1) for b-algebras. In the paper's final section, we demonstrate how <<far from>> betweenness the ternary relation $R$ is.

\section{Notation and first definitions.}\label{sec:notation-et-al}

\subsection{Boolean algebras}
Throughout, $\klam{A,+, \cdot, -, \zero, \one}$ is a non-trivial Boolean algebra, usually denoted by $A$. For $a,b \in A$ we define the \emph{symmetric sum} 
\begin{equation}\tag{$\mathtt{df}\,\mathord{\symsum}$}
a \symsum b \df a \cdot b + -(a + b), 
\end{equation}
which is the complement of the symmetric difference $(a \cdot -b) + (b \cdot -a)$ and corresponds to logical equivalence. Clearly, $a = b$ \tiff $a \symsum b = \one$.
\begin{lemma} [{\citealp[Lemma 5.22]{kop89}}]\label{lem:Bcong}
If $F$ is a filter on $A$, then the relation on $A$ defined by 
\[
\text{$a \mathrel{\theta_F} b$ \tiff $a \symsum b \in F$}
\]
is a Boolean congruence relation. Conversely, if $\theta$ is a Boolean congruence on $A$, then $F_\theta \df \set{a \in A: a \mathrel{\theta} \one}$ is a filter on $A$. Furthermore, $\theta_{F_\theta} = \theta$ and $F_{\theta_F} = F$.
\end{lemma}

\subsection{PS-algebras}

A binary \emph{possibility operator on $A$} is a mapping $f: A \times A \to A$ which satisfies for all $a,b,a',b' \in A$
\begin{xalignat}{2}
&f(a, \zero) = f(\zero,a) = \zero, && \text{Normality}\label{eq:f-normality} \\
&f(a,b) + f(a,b') = f(a,b+b'), &&\text{Additivity}, \\
&f(a,b) + f(a',b) = f(a+a',b). \nonumber
\end{xalignat}
The \emph{dual $\fd$ of $f$} is the mapping defined by $\fd(a,b) \df -f(-a,-b)$, often called a binary\emph{ necessity operator}. 

A binary \emph{sufficiency operator on $A$} is a mapping $g: A \times A \to A$ which satisfies for all $a,b,a',b' \in A$
\begin{xalignat}{2}
&g(a, \zero) = g(\zero,a) = \one, && \text{Co-normality} \\
&g(a,b) \cdot g(a,b') = g(a,b+b'), &&\text{Co-additivity},\label{eq:g-coadditivity} \\
&g(a,b) \cdot g(a',b) = g(a+a',b). \nonumber
\end{xalignat}
It is well known that $f$ is isotone in each component and that $g$ is antitone in each component. A \emph{PS-algebra} is a structure $\frA \df \klam{A,f,g}$ where $f$ is a possibility operator and $g$ is a sufficiency operator.

We prepare the description of the congruences on a PS-algebra $\frA$ with a few known facts. The Boolean congruences on $A$ were described earlier in Lemma \ref{lem:Bcong}. 
Congruences with respect to a binary necessity operator were characterized in \citet{ven07}.

\begin{lemma}[{\citealp[Theorem 29]{ven07}}]\label{lem:NecCong} 
If $m$ is a binary necessity operator on $\frA$ and $F$ is a filter of $A$, then $\theta_F$ preserves $m$ \tiff $a \in F$ implies $m(a, \zero) \cdot m(\zero,a) \in F$.
\end{lemma}
To apply these results, we first define for a sufficiency operator $g: A \times A \to A$ a mapping $g^*: A \times A \to A$ by $g^*(a,b) = g(-a,-b)$. The operator $g^*$ is a binary necessity operator, and a Boolean congruence preserves $g$ \tiff it preserves $g^*$, since $g$ and $g^*$ are mutually term definable. We now define an auxiliary mapping  $u\colon A \times A \to A$\footnote{The name refers to the universal modality which is related to the unary discriminator, see e.g. \cite{gpt87,gp92}.} 
\begin{equation}\tag{$\mathtt{df}\,u$}\label{def:u}
u(a,b) = f^\partial(a,b) \cdot g^*(a,b).
\end{equation}
Clearly, $u$ is a binary necessity operator on $A$.

Given a Boolean algebra $A$ with binary necessity operators $m_i$ (for some index set $I$), a filter of $F$ of $A$ is a \emph{congruence filter} \tiff $\theta_F$ preserves each $m_i$.

\begin{theorem}\label{th:congruence-filter-via-u}
Let $\frA\defeq\klam{A,f,g}$ be a PS-algebra. Suppose that $F$ is a filter of $A$. Then, $F$ is a congruence filter (i.e., $\theta_F$ preserves $f^\partial$ and $g^\ast$) \tiff   $u(a, \zero) \cdot u(\zero,a) \in F$ for all $a \in F$.
\end{theorem}
\begin{proof}
\aright Let $F$ be a congruence filter and $a \in F$. Then, $f^\partial(a,\zero) \cdot f^\partial(\zero,a) \in F$ and $g^*(a,\zero) \cdot g^*(\zero,a) \in F$ by Lemma \ref{lem:NecCong}. Thus,
\begin{multline*}
u(a, \zero) \cdot u(\zero,a) =  f^\partial(a,\zero) \cdot g^*(a,\zero) \cdot f^\partial(\zero,a) \cdot g^*(\zero,a) = \\
f^\partial(a,\zero)\cdot f^\partial(\zero,a)  \cdot g^*(a,\zero) \cdot g^*(\zero,a) \in F.
\end{multline*}

\smallskip

\aleft If $a \in F$ and $u(a, \zero) \cdot u(\zero,a) \in F$, then 
\begin{gather*}
f^\partial(a,\zero)  \cdot g^*(a,\zero) \cdot f^\partial(\zero,a) \cdot g^*(\zero,a) \in F,
\end{gather*}
and thus, $f^\partial(a,\zero) \cdot f^\partial(\zero,a) \in F$ and $g^*(a,\zero) \cdot g^*(\zero,a) \in F$ since $F$ is a filter.
\end{proof}

\subsection{Complex algebras and canonical frames}\label{subsec:complex-algebras}

The \emph{canonical frame} of $\frA$ is the structure $\Cf(\frA) \df \langle\Ult A,R_f,S_g\rangle$, where $R_f,S_g$ are ternary relations on $\Ult A$ defined by
\begin{align}
R_f(\ult_1,\ult_2,\ult_3) &\iffdef f[\ult_1 \times \ult_3] \subseteq \ult_2,\tag{$\dftt{R_f}$} \\
S_g(\ult_1,\ult_2,\ult_3) &\iffdef g[\ult_1\times\ult_3]\cap\ult_2\neq\emptyset.\tag{$\dftt{S_g}$}
\end{align}

Additionally, we define an auxiliary operator $h\colon A \times A \to A$ by 
\begin{equation}\tag{$\mathtt{df}\,h$}
    h(a,b)\defeq f(a,b)+-g(a,b), \label{def:h}
\end{equation}
and we set
\begin{gather}
T_h(\ult_1,\ult_2,\ult_3) \iffdef h[\ult_1 \times \ult_3] \subseteq \ult_2. \tag{$\dftt{T_h}$} 
\end{gather}

\begin{lemma}\label{lem:TRS}
$T_h = R_f \cup -S_g$.
\end{lemma}
\begin{proof}
($\subseteq$)   Suppose $h[\ult_1\times\ult_3]\subseteq\ult_2$ and $g[\ult_1\times\ult_3]\cap\ult_2\neq\emptyset$. Let $\klam{a,b}\in \ult_1\times\ult_3$ be such that $g(a,b)\in\ult_2$. Pick $x\in\ult_1$ and $y\in\ult_3$. As $x\cdot a\in\ult_1$ and $y\cdot b\in\ult_3$, it must hold that $h(x\cdot a,y\cdot b)\in\ult_2$. As $g$ is antitone, also $g(x\cdot a,y\cdot b)\in\ult_2$. Thus by the definition of~$h$ we obtain $f(x\cdot a,y\cdot b)\in\ult_2$. Now, $f$ is isotone, so $f(x,y)\in\ult_2$, as required.

\smallskip

($\supseteq$) Let $\ult_1,\ult_2,\ult_3 \in \Ult A$, and $a \in \ult_1, b \in\ult_3$. If $R_f(\ult_1,\ult_2,\ult_3)$, then, $f(a,b) \in \ult_2$, and it follows that $h(a,b) = f(a,b)+-g(a,b) \in \ult_2$. If $\klam{\ult_1,\ult_2,\ult_3 }$ is not an element of $S_g$, then $g[\ult_1\times\ult_3]\cap\ult_2 = \emptyset$, hence, $g(a,b) \not\in \ult_2$ which implies $-g(a,b) \in \ult_2$. It follows that $h(a,b) = f(a,b)+-g(a,b) \in \ult_2$.
\end{proof}

In this paper, a \emph{frame} is a structure $\frF \df \klam{W,R,S}$ where $W$ is a nonempty set, and $R,S$ are ternary relations on $W$. The \textit{full complex algebra of $\frF$} is the structure $\Cm(\frF) \df \klam{2^W, \poss{R}, \suff{S}}$ where $\poss{R}$ and $\suff{S}$ are binary operators (of possibility and sufficiency, respectively) on $2^W$ defined by
\begin{align}
\poss{R}(X,Y) &\df \set{z \in W: (X \times \set{z} \times Y) \cap R \neq \emptyset},\tag{$\mathtt{df}\,\poss{R}$} \\
\suff{S}(X,Y) &\df \set{z \in W: (X \times \set{z} \times Y) \subseteq S}. \tag{$\mathtt{df}\,\suff{S}$}
\end{align}

A subalgebra of $\Cm(\frF)$ is called a \emph{complex algebra}.

\section{The binary logic \texorpdfstring{$\Kpm$}{K\#}}\label{sec:logic}

We will develop the logic \Kpm for the binary case using the copying construction from \cite{vak89} adapted for our needs. This is a Boolean logic with a set $\Var$  of propositional variables, a constant $\top$, and two extra binary modalities $\Diamond$ and $\win$ with duals $\Box$ and $\diamondvert$. Formulas have the form
\begin{gather*}
\top,\ p,\ \neg \phi,\ \phi \land \psi,\ \Diamond(\phi,\psi),\ \win(\phi,\psi),
\end{gather*}
where $p \in \Var$. We use the usual definitions of the Boolean connectives $\bot$, $\lor$, $\rightarrow$, $\leftrightarrow$. The axioms are those for classical propositional logic plus the following modal axioms

\pagebreak

\begin{gather}
\neg\Diamond(\bot, \phi)\qquad\neg\Diamond(\phi, \bot), \label{M1}\tag{M1}\\
\begin{split}
\Diamond(\phi, \psi) \lor \Diamond(\phi, \chi) &{}\leftrightarrow \Diamond(\phi, \psi \lor \chi)\\  
\Diamond(\phi, \psi) \lor \Diamond(\chi, \psi ) &{}\leftrightarrow \Diamond(\phi \lor \chi, \psi), 
\end{split}\label{M2}\tag{M2}\\
\win(\bot, \phi)\qquad \win(\phi, \bot), \label{M3}\tag{M3}\\
\begin{split}
\win(\phi, \psi) \land \win(\phi, \chi) &{}\leftrightarrow \win(\phi, \psi \lor \chi)\\  
\win(\phi, \psi) \land \win(\chi, \psi ) &{}\leftrightarrow \win(\phi \lor \chi, \psi).
\end{split}\label{M4}\tag{M4}
\end{gather}

The auxiliary operator $\possu{U}$ is defined by 
\begin{gather}\tag{$\dftt{\possu{U}}$}
\possu{U}(\phi,\psi) \df \Diamond(\phi,\psi) \lor \neg \win(\phi,\psi).
\end{gather}
Its dual is denoted by $\necu{U}$:
\[
\necu{U}(\varphi,\psi)\defeq\Box(\varphi,\psi)\wedge \win(\neg\varphi,\neg\psi).
\]

 We consider the following additional axioms,
\begin{align}
&\possu{U}(\phi,\top) \land \possu{U}(\psi, \top)\rarrow \possu{U}(\phi,\psi),\label{CU}\tag{$\mathrm{C}_U$}\\
&\phi \rightarrow \possu{U}(\phi,\top),\label{TU}\tag{$\mathrm{T}_U$}\\
&\possu{U}(\possu{U}(\phi,\top),\top) \rightarrow  \possu{U}(\phi,\top), \label{4U}\tag{$4_U$}\\
&\phi \rightarrow \necu{U}\left(\possu{U}(\phi,\top),\bot\right). \label{BU}\tag{$\mathrm{B}_U$}
\end{align}
The names \eqref{TU}, \eqref{4U} and \eqref{BU} are motivated in an obvious way by reference to the standard unary modal axioms corresponding to reflexivity, transitivity, and symmetry. The name for \eqref{CU} is justified by the axiom postulating a~non-standard behavior of possibility with respect to conjunction.

\Kpm is the smallest set of formulas that contains propositional tautologies, all instances of \eqref{M1}--\eqref{M4}, \eqref{CU}--\eqref{BU}, and is closed under modus ponens, uniform substitution, and the following rules for $\Diamond$ and $\win$:
\begin{equation*}
    \inference{\vdash\varphi\rarrow\psi}{\vdash\Diamond(\varphi,\chi)\rarrow\Diamond(\psi,\chi)}\quad\quad\inference{\vdash\varphi\rarrow\psi}{\vdash\Diamond(\chi,\varphi)\rarrow\Diamond(\chi,\psi)}
\end{equation*}
\par\begin{equation*}
    \inference{\vdash\varphi\rarrow\psi}{\vdash\win(\varphi,\chi)\rarrow\win(\psi,\chi)}\quad\quad\inference{\vdash\varphi\rarrow\psi}{\vdash\win(\chi,\varphi)\rarrow\win(\chi,\psi)}\,.
\end{equation*}

We have that
\begin{align}
      \Kpm\vdash{}& \possu{U}(\phi,\psi)\rarrow\possu{U}(\psi,\phi).\tag{$\mathrm{S}_U$}\label{eq:SU}
\end{align}

\subsection{Relational semantics}

$\frF\defeq\klam{W,R,S}$ is called a \emph{weak MIA frame} (\emph{wMIA frame}, for short), if $S \subseteq R$. The class of weak MIA frames is decisive in the determination of the relational models of the logic $\Kpm$.

Generalizing the notion of complex algebras of weak MIA frames we arrive at the following definition. A PS-algebra $\frA$ is a \emph{weak MIA} (wMIA) if it satisfies
\begin{equation}\label{wMIA}
    (\forall a,b\in A)\,[a \neq \zero \tand b \neq \zero \Rarrow g(a,b) \leq f(a,b)]. \tag{wMIA}
\end{equation}

The class of weak MIAs is denoted by \wmia. Note that \wmia is a first-order positive universal class, and so it is closed under subalgebras and homomorphic images. The following establishes a relationship between weak MIAs and weak MIA frames.

\begin{lemma}[\citealp{dgm23}]\label{lem:repwMIA} 
\begin{enumerate}
\item The complex algebra of a wMIA frame is a weak MIA.
\item The canonical frame of a weak MIA is a weak MIA frame.
\end{enumerate}
\end{lemma}

\begin{theorem}[\citealp{dgm23}]\label{thm:repwMIA} 
\begin{enumerate}
\item Each wMIA frame is embeddable into the canonical frame of its complex algebra.
\item Each weak MIA is embeddable into the complex algebra of its canonical frame.
\end{enumerate}
\end{theorem}

Notation not explained here can be found in  \cite[Section 5]{brv_modal}, \cite{dgm23}, and \cite{bs_ua}, in particular, II~\S 10.

Models are structures $\frM \df \klam{\frF,v}$ where $\frF\defeq\klam{W,R,S}$ is a wMIA frame. and  $v\colon \Var \to 2^W$ is  a valuation, extended over the Boolean connectives in the usual way, and over the modal operators as follows:
\begin{align}
v(\Diamond(\phi,\psi)) &\df \set{x: (v(\phi) \times \set{x} \times v(\psi)) \cap R \neq \emptyset}, \\
v(\win(\phi,\psi)) &\df \set{x: (v(\phi) \times \set{x} \times v(\psi)) \subseteq S}.
\end{align}
We call $\frF$ the \textit{base of the model}. A formula $\phi$ \textit{is satisfied at $x \in W$ with respect to $\frM$}, written as $\frM,x \models \phi$, if $x \in v(\phi)$. Then, 
\begin{align}
\frM, x \models \Diamond(\phi,\psi)  &{}\Iff (v(\phi) \times \set{x} \times v(\psi)) \cap R \neq \emptyset,\\
\frM,x \models \win(\phi,\psi) &{}\Iff v(\phi) \times \set{x} \times v(\psi) \subseteq S.
\end{align}
Consequently,
\begin{align}
    \frM,x\models \diamondvert(\varphi,\psi)\Iff{}& (v(\neg\varphi)\times\{x\}\times v(\neg\psi))\cap-S\neq\emptyset,\\[1em]
    \begin{split}
        \frM,x\models \possu{U}(\varphi,\psi)\Iff{}&(v(\phi) \times \set{x} \times v(\psi)) \cap R \neq \emptyset \tor\\
        &\phantom{aaaaaa}(v(\phi) \times \set{x} \times v(\psi)) \cap-S \neq \emptyset.
    \end{split}
\end{align}
Thus, 
\begin{align}
    \frM,x\models \possu{U}(\varphi,\psi)&{}\Iff v(\phi) \neq\emptyset\tand v(\psi)\neq\emptyset,\label{eq:possu-non-empty}\\
    \frM,x\models \necu{U}(\varphi,\psi)&{}\Iff v(\phi) = W\tor v(\psi)=W.\label{eq:necu-total}
\end{align}
and---thanks to these---with $\possu{U}$ and $\necu{U}$ we can associate global modalities, existential and universal, respectively
\begin{align}
   \frM, x \models \possu{U}(\phi,\top)& {}\Iff (\exists y\in W)\,\frM,y\models \phi,\tag{$\mathsf{E}_U$}\label{eq:AU}\\ 
    \frM, x \models \necu{U}(\phi,\bot)& {}\Iff (\forall y\in W)\,\frM,y\models \phi. \tag{$\mathsf{A}_U$}
\end{align}

\subsection{The canonical frame for \texorpdfstring{$\boldsymbol{\Kpm}$}{K}}
The main objectives for this section are proofs of soundness and completeness of $\Kpm$ with respect to the class of wMIA frames. For these, we use the standard notion of validity in a frame, i.e., a formula $\varphi$ is \emph{valid} in $\frF$ if for any model $\frM$ based on $\frF$ and any $x\in W$, $\frM,x\models\varphi$. 

Soundness, as usual, is easy.
\begin{theorem}[Soundness]\label{thm:Kpm-soundness}
\Kpm is sound with respect to wMIA frames.    
\end{theorem}
\begin{proof}
Let $\frM\defeq\klam{\frF,v}$ be a model with the base $\frF\defeq\klam{W,R,S}$. It is easy to see that the axioms \eqref{M1}--\eqref{M4} are satisfied at all $x\in W$.

Straightforward computations show that for all $x\in W$, $\frM, x\models\eqref{TU}$, $\frM,x\models\eqref{4U}$,  and $\frM,x\models\eqref{BU}$. 
For \eqref{CU}, observe that if $x\in W$ is such that $\frM, x\models \possu{U}(\phi,\top)\land \possu{U}(\psi,\top)$, then there exist $y,z\in W$ such that $\frM, y\models \phi$ and $\frM, z\models \psi$, i.e., $v(\phi)\neq\emptyset\neq v(\psi)$.  Therefore, $\frM,x\models \possu{U}(\phi,\psi)$ by \eqref{eq:possu-non-empty}.

Since the valuation is arbitrary, we have shown that the axiomatization of $\Kpm$ is sound with respect to wMIA frames.
\end{proof}

For completeness, we will need several definitions and lemmas. The key notion will be that of the canonical frame of the logic \Kpm.

Let $W^{\Kpm}$ be the collection of all maximal consistent theories of \Kpm. The \emph{canonical frame}\footnote{The terminology we use is ambiguous, but intentionally so. In subsection \ref{subsec:complex-algebras} we use the term for frames of ultrafilters of Boolean algebras with operators. Here, we use it for the relational structure whose domain consists of maximal consistent theories of $\Kpm$. The two notions can be reconciled, though, as the canonical frame for $\Kpm$ is indeed a frame of ultrafilters of the Lindenbaum algebra of $\Kpm$.} for \Kpm is the structure 
\[
\frF^{\Kpm} \df \klam{W^{\Kpm}, R^{\Kpm}, S^{\Kpm}}
\]
where 
\begin{align}
R^{\Kpm}(x,y,z) &\iffdef \Diamond[x \times z] \subseteq y,\tag{$\mathtt{df}\,R^{\Kpm}$} \\
S^{\Kpm}(x,y,z) &\iffdef \win[x \times z] \cap y \neq \emptyset.\tag{$\mathtt{df}\,S^{\Kpm}$}
\end{align}
We also define
\begin{gather}\tag{$\mathtt{df}\,T^{\Kpm}$}
T^{\Kpm}(x,y,z) \iffdef \possu{U}[x \times z] \subseteq y.    
\end{gather}

In analogy to Lemma \ref{lem:TRS} it can be shown that 
\begin{lemma}\label{lem:T=R+-S-in-can-mod} In the canonical frame $\frF^{\Kpm}$, $T^{\Kpm} =R^{\Kpm}\cup-S^{\Kpm}$.
\end{lemma}

To simplify the notation we usually omit the superscripts in those situations which will not lead to any ambiguities.

 We will prove that 
\begin{lemma}\label{lem:T-properties}
In $\frF^{\Kpm}$ the relation $T$ satisfies the following properties\/\textup{:}
\begin{enumerate}[label=\textup{(T\arabic*)},itemsep=2pt]
    \item $T(x,x,x)$.\label{T1}
    \item If $T(x,y,z)$, then $T(z,y,x)$.\label{T2}
    \item If $T(x_1,y,z_1)$, $T(x_2,y,z_2)$ and $T(x_3,y,z_3)$, then $T(a,b,c)$ for every triple $\klam{a,b,c}$ in $\{x_1,x_2,x_3\}^3\cup\{z_1,z_2,z_3\}^3$.\label{T3}
\end{enumerate}
\end{lemma}
\begin{proof}
For \ref{T1}, consider $\phi,\psi\in x$. So $\phi\land\psi\in x$ and, since $x$ is a maximal theory, by \eqref{TU} and \eqref{CU} we obtain $\possu{U}(\phi\land\psi,\phi\land\psi)\in x$. It follows that $\possu{U}(\phi,\psi)\in x$, and therefore $\possu{U}[x\times x]\subseteq x$.

\smallskip

For \ref{T2}, suppose that $\possu{U}[x\times z]\subseteq y$. Let $\phi\in z$ and $\psi\in x$. By assumption, $\possu{U}(\psi,\phi)\in y$, and thus  $\possu{U}(\phi,\psi)\in y$ by \eqref{eq:SU} as required.

\smallskip

For \ref{T3}, suppose that $\possu{U}[x_1\times z_1]\subseteq y$, $\possu{U}[x_2\times z_2]\subseteq y$ and $\possu{U}[x_3\times z_3]\subseteq y$. Let $\phi\in x_1$ and $\psi\in x_3$. By assumption, we know that $\possu{U}(\phi,\top)\in y$ and $\possu{U}(\psi,\top)\in y$. By \eqref{BU}
\[
\necu{U}(\possu{U}(\possu{U}(\psi,\top),\top),\bot)\in y
\]
and by \eqref{4U}
\[
\necu{U}(\possu{U}(\psi,\top),\bot)\in y.
\]
By the same argument $\necu{U}(\possu{U}(\phi,\top),\bot)\in y$. Since 
\[
\necu{U}(\possu{U}(\phi,\top),\bot)\land \necu{U}(\possu{U}(\psi,\top),\bot) \Iff \necu{U}(\possu{U}(\phi,\top)\land \possu{U}(\psi,\top),\bot),
\]
from this and \eqref{CU} it follows that 
\[
\necu{U}(\possu{U}(\phi,\psi),\bot)\in y,\quad\text{i.e.,}\quad \neg \possu{U}(\neg\possu{U}(\phi,\psi),\top)\in y. 
\]
In consequence $\neg\possu{U}(\phi,\psi)\notin x_1\cup x_2\cup x_3$ because otherwise
\[
\possu{U}(\neg\possu{U}(\phi,\psi),\top)\in y,
\]
a~contradiction. As we are dealing with maximal consistent theories of $\Kpm$, we obtain $\possu{U}(\phi,\psi)\in x_1\cap x_2\cap x_3$, and therefore $T(x_1,x_1,x_3)$, $T(x_1,x_2,x_3)$ and $T(x_1,x_3,x_3)$.

All remaining possibilities are shown using the previous two points and choosing $\varphi$ and $\psi$ suitably.
\end{proof}

\begin{lemma}\label{lem:T-Wx-for-completeness}
    Let $\frF^{\Kpm}$ be the canonical frame for $\Kpm$. For $x\in W$, let
    \[
        W_x\defeq\{y\in W\mid (\exists z\in W)\,T(y,x,z)\}.
    \]
    If $y\in W_x$ and $T(a,y,b)$, then $a,b\in W_x$.
\end{lemma} 
\begin{proof}
    Assume $z\in W$ is such that $T(y,x,z)$ and that $T(a,y,b)$. By \ref{T1} we have that $T(x,x,x)$, and from \ref{T3} we obtain $T(x,y,x)$. Applying \ref{T3} again, but this time to $T(a,y,b)$ and $T(x,y,x)$, we obtain that $T(a,x,a)$ and $T(b,x,b)$, so both $a$ and $b$ are in $W_x$.
\end{proof}
\begin{lemma}\label{lem:subframe-lemma}
        If $x\in W^{\Kpm}$, then $\frF_x\defeq\klam{W_x,R_x,S_x}$ such that
        and $R_x\defeq R\cap W_x^3$ and $S_x\defeq S\cap W_x^3$ is a wMIA frame (and a subframe\footnote{$\frF'\defeq\klam{W',R',S'}$ is a \emph{subframe} of $\frF\defeq\klam{W,R,S}$ if $W'\subseteq W$ and $R'$ and $S'$ are restrictions of, respectively, $R$ and $S$ to $W'$.} of $\frF^{\Kpm}$).
    \end{lemma}
\begin{proof}
    Let $\klam{r,s,t} \in S_x$. By the definition of $W_x$ and by \ref{T3} we obtain $T(r,s,t)$, so $\klam{r,s,t}\in R\cup -S$ by Lemma~\ref{lem:T=R+-S-in-can-mod}. By assumption, $\klam{r,s,t}\in S$ so we obtain $\klam{r,s,t}\in R\cap W_x^3$, i.e., $\klam{r,s,t}\in R_x$, as required.
\end{proof}

\begin{theorem}[Completeness]\label{th:completeness-Kpm}
If $\Kpm\nvdash\delta$, then there is a model of 
$\Kpm$ based on a wMIA subframe of $\frF^{\Kpm}$ at which $\delta$ fails. Thus, if $\delta$ is valid in the class of wMIA frames, then $\Kpm\vdash\delta$.
\end{theorem}
\begin{proof}
Suppose that $\Kpm \not\vdash\delta$, and let $x$ be a maximal consistent theory of the logic with $\delta\notin x$. Let $V\colon\Fml\to 2^{W^{\Kpm}}$ be the following mapping on the canonical frame $\frF^{\Kpm}$: $y\in V(\varphi)$ \tiff $\varphi\in y$. So $x\notin V(\delta)$. Take the weak MIA subframe $\frF_x$ and consider the function $v\colon \Fml\to 2^{W_x}$ defined by
\[v(\psi)\df V(\psi)\cap W_x.\]
Let us show that $v$ is a valuation on the frame $\frF_x$, and in consequence of this and Theorem~\ref{thm:Kpm-soundness}, $\frM_x\defeq\klam{\frF_x,v}$ is a model of $\Kpm$ (in which $\delta$ fails at $x$, of course).

For the Boolean connectives, the proofs are straightforward, e.g., for negation, we have
\[
    v(\neg\varphi)\defeq V(\neg\varphi)\cap W_x=[W\setminus V(\varphi)]\cap W_x=W_x\setminus v(\varphi).
\]
For the possibility operator
\begin{align*}
y\in v(\Diamond(\varphi,\psi))&{}\iffdef y\in W_x\cap V(\Diamond(\varphi,\psi))\\
&{}\iffdef V(\varphi)\times\{y\}\times V(\psi)\cap R\neq\emptyset\tand y\in W_x
\end{align*}
so, from the first conjunct, we obtain there are $a\in V(\varphi)$ and $b\in V(\psi)$ such that (i) $R(a,y,b)$ and from the second one that (ii) $y\in W_x$. From (i) and Lemma~\ref{lem:T=R+-S-in-can-mod} it follows that $T(a,y,b)$, so (ii) and Lemma~\ref{lem:T-Wx-for-completeness} entail that $a,b\in W_x$. Thus,
\begin{align*}
   y\in v(\Diamond(\varphi,\psi))&{}\Iff (V(\varphi)\cap W_x)\times\{y\}\times (V(\psi)\cap W_x)\cap R\neq\emptyset\tand y\in W_x\\
   &{}\Iff v(\varphi)\times\{y\}\times v(\psi)\cap R_x\neq\emptyset.
\end{align*}
For the sufficiency operator we have to demonstrate that
\[
\{y\in W:V(\varphi)\times\{y\}\times V(\psi)\subseteq S\}\cap W_x=\{y\in W_x:v(\varphi)\times\{y\}\times v(\psi)\subseteq S_x\}.
\]

($\subseteq$) If $V(\varphi)\times\{y\}\times V(\psi)\subseteq S$, $y\in W_x$ and $\klam{a,y,b}\in v(\varphi)\times\{y\}\times v(\psi)$, then also $a,b\in W_x$ and thus the triple is in $S_x$ as required.

\smallskip

($\supseteq$) Let $y\in W_x$ be such that $v(\varphi)\times\{y\}\times v(\psi)\subseteq S_x$.  By the definition of $W_x$ we have that there is $z\in W$ such that $T(y,x,z)$. Take $\klam{a,b,c}\in V(\varphi)\times\{y\}\times V(\psi)$, and assume towards a contradiction that $\klam{a,y,b}\notin S$. By Lemma~\ref{lem:T=R+-S-in-can-mod} it is the case that $T(a,y,b)$. By Lemma~\ref{lem:T-Wx-for-completeness}, $a,b\in W_x$, so $\klam{a,y,b}\in S_x\subseteq S$, which contradicts our assumption.
\end{proof}

\subsection{Special frames and models}\label{sub:special-model}

A weak MIA frame $\frF\defeq\klam{W,R,S}$ is called \emph{special} if $R=S$. Similarly, a model $\frM$ based on a wMIA frame $\frF$ is special if $\frF$ itself is special.

Two models $\frM \df \klam{W,R,S,v}$ and $\frM' \df \klam{W',R',S',v'}$ are called \emph{modally equivalent} if the same formulas are true in both, i.e., if for all $\phi \in \Fml$,
\[
\frM \models \phi \Iff \frM' \models \phi.\footnotemark
\]
\footnotetext{$\frM \models \phi$ if  $\frM,x\models\varphi$ for every point $x$ of $\frM$.}

Our next aim is to construct---given $\klam{W,R,S,v}$---a special model $\klam{\underline{W}, \uR, \uS, \uv}$, that is modally equivalent to $\klam{W,R,S,v}$.  The idea of the construction of $\uR$ and $\uS$ is to ``separate'' triples $\klam{x,y,z}$ which are in the intersection of $R$ and $-S$, i.e. to remove those triples which prevent $R = S$ (see Figure~\ref{fig:special-model}).

Let us start with a fixed frame $\frF\defeq\klam{W,R,S}$. Let $\frF\uplus\frF'$ be the disjoint union of~$\frF$ with its isomorphic copy $\frF'\defeq\klam{W',R',S'}$ whose domain is disjoint from that 
of~$\frF$. Thus, 
\[
\frF\uplus\frF'\defeq\klam{W\cup W',R\cup R',S\cup S'}.\footnotemark
\]
\footnotetext{The copying method was introduced in \cite{vak89}.}
Clearly, if both frames are wMIA frames, then $S\cup S'\subseteq R\cup R'$, so the disjoint union is a wMIA frame as well.

For $\underline{W}\defeq W\cup W'$, there are triples in $\underline{W}^3$ such that
\begin{enumerate}[label=(\arabic*)]
    \item all coordinates are from either $W$ or $W'$, which we call \emph{pure}, or
    \item at least one coordinate is from $W$ and at least one from $W'$; these triples will be called \emph{mixed}.
\end{enumerate}

For a fixed pure triple $\klam{x,y,z}\in W^3$ (or its copy $\klam{x',y',z'}\in (W')^3$), in $\underline{W}^3$ there are eight ``variants'' of  $\klam{x,y,z}$: the triple itself and 
\begin{gather*}
\klam{x',y,z},\ \klam{x,y',z},\ \klam{x,y,z'},\ \klam{x',y',z},\ \klam{x',y,z'},\ \klam{x,y',z'},\ \klam{x',y',z'}.
\end{gather*}
We denote the collection of these by
\[
c(x,y,z)\defeq\{\klam{r,s,t}: r\in\{x,x'\}\tand s\in\{y,y'\}\tand t\in\{z,z'\}\}
\]
and we call it the \emph{cell} determined by $\klam{x,y,z}$. 

Let $i\colon W\to W'$ be the bijection that assigns to any element of the domain its copy $i(u)\defeq u'$. Let $i^{\ast}$ be the function inverse to $i$. By means of these we define an onto function $j\colon W\cup W'\to W$ such that
\[
j(t)\defeq
\begin{cases}
    t\quad&\text{if $t\in W$,}\\
    i^{\ast}(t)\quad&\text{if $t\in W'$.}
\end{cases}
\]
Clearly, for $x\in W$, $j^{-1}(x)=\{x,x'\}$ and $\underline{W}=\bigcup_{x\in W}j^{-1}(x)$. 
Moreover
\[
c(x,y,z)=\{\klam{r,s,t}: x=j(r)\tand y=j(s)\tand z=j(t)\}.
\]

The next lemma justifies the name \emph{cell}.

\begin{lemma}[Partition Lemma]\label{lem:part}
The collection
\[
\mathop{\mathrm{Cell}}(\underline{W}^3)\defeq\set{c(x,y,z): \klam{x,y,z} \in W^3} 
\]
is a partition of $\underline{W}^3$.
\end{lemma}
\begin{proof}
Let $\klam{r,s,t} \in \underline{W}^3$. Then, there are $x,y,z \in W$ such that 
$x=j(r)$ and $y=j(s)$ and $z=j(t)$, and thus $\klam{r,s,t}\in c(x,y,z)$.

If $\klam{r,s,t}\in c(x_0,y_0,z_0)\cap c(x_1,y_1,z_1)$, then
\[
x_0=j(r)=x_1\qtand y_0=j(s)=y_1\qtand z_0=j(t)=z_1
\]
and so $c(x_0,y_0,z_0)=c(x_1,y_1,z_1)$.
\end{proof}

For $P\subseteq W^3$, let $P'$ be its copy. $P$ and $P'$ are relations that contain pure triples only. For relations that are subsets of $\uW^3$ (and may therefore contain mixed triples), we are going to use underlined letters, e.g., $\uP$, $\uR$, $\uS$.

Given a relation $\uP\subseteq \uW^3$, we define the \emph{mixture} of $\uP$ as the value of the function $m\colon 2^{\uW^3}\to2^{\underline{W}^3}$ defined by
\begin{equation}\tag{$\mathrm{df}\,m$}
    m(\uP)\defeq\bigcup_{\klam{r,s,t}\in \uP}c(j(r),j(s),j(t)).
\end{equation}
In consequence, for $P\subseteq W^3$ and its copy $P'\subseteq W'^3$ we have
\[
m(P)=\bigcup_{\klam{x,y,z}\in P} c(x,y,z)\qtand m(P')=\bigcup_{\klam{x',y',z'}\in P'} c(x,y,z).
\]
\begin{lemma}\label{lem:m-properties}
For all relations $P$ and $Q$ from $W^3$ we have:
\begin{align}
m(P)&{}=m(P')=m(P\cup P'),\label{eq:m-any-P}\\
m(W^3)&{}=\uW^3,\label{eq:m-unity}\\
P\subsetneq m(P)\quad&{}\tand\quad P'\subsetneq m(P'),\label{eq:m-closure}\\
P\subseteq Q&{}\Rarrow m(P)\subseteq m(Q),\label{eq:m-isotone}\\
\uW^3\setminus m(P)=m(W^3\setminus P)\quad&{}\tand\quad \uW^3\setminus m(P')=m(W'^3\setminus P')\label{eq:m-complement}\\
m(P\cap Q)&{}=m(P)\cap m(Q),\label{eq:m-intersection}\\
m(P\cup Q)&{}=m(P)\cup m(Q).\label{eq:m-union}
\end{align}
\end{lemma}
\begin{proof}
\eqref{eq:m-any-P}, \eqref{eq:m-closure}, \eqref{eq:m-isotone}, and \eqref{eq:m-union} have routine proofs that we leave to the reader.

\smallskip

\eqref{eq:m-unity} By the Partition Lemma.

\smallskip

  \eqref{eq:m-complement} ($\subseteq$) Let $\tau$ be a triple not in $m(P)$. Thus, for all $\klam{x,y,z}\in P$, $\tau\notin c(x,y,z)$. By the Partition Lemma there is a triple $\klam{x_0,y_0,z_0}\in W^3$ such that $\tau\in c(x_0,y_0,z_0)$. Clearly, $\klam{x_0,y_0,z_0}\notin P$, so $\tau\in m(W^3\setminus P)$.

  \smallskip

  $(\supseteq)$ If $\tau\in m(W^3\setminus P)$, then there is $\klam{x,y,z}\notin P$ such that $\tau\in c(x,y,z)$. If there were $\klam{r,s,t}\in m(P)$ such that $\tau\in c(j(r),j(s),j(t))$, then $P$ and $W^3\setminus P$ would have a non-empty intersection, which is impossible.

  The proof for the copy of $P$ is similar.

    \smallskip

\eqref{eq:m-intersection} The left-to-right inclusion holds as $m$ is isotone. For the reverse inclusion, pick $\tau\in m(P)\cap m(Q)$. Then, there are triples $\klam{x_0,y_0,z_0}\in P$ and $\klam{x_1,y_1,z_1}\in Q$ such that
\[
    \tau\in c(x_0,y_0,z_0)\cap c(x_1,y_1,z_1).
\]
    By the Partition Lemma, these triples must be identical, and so $\klam{x_0,y_0,z_0}\in P\cap Q$, which entails that $\tau\in m(P\cap Q)$.
\end{proof}

The frame $m(\frF)\defeq\klam{\underline{W},m(R),m(S)}$---obtained from the disjoint union of $\frF$ and its copy $\frF'$---will be called the \emph{mixture} of $\frF$. By \eqref{eq:m-isotone}, it is a wMIA frame, provided that $\frF$ itself is.

The mixture $m(\frF)$ of any wMIA frame $\frF$ can be turned into a special frame via the following construction of $\uR$ and $-\uS$ (below, $-R$ and $-S$ are the complements of the two relations in $W^3$):
\begin{align*}
\uR \df{}& m(S)\cup[(m(R)\cap m(-S))\setminus(R\cup R')]\\
={}&m(S)\cup[m(R\setminus S)\setminus(R\cup R')],\\[5pt]
-\uS \df{}& m(-R)\cup[(R\cup R')\cap(-S\cup(-S)')]\\
={}&m(-R)\cup[(R\cup R')\setminus(S\cup S')].
\end{align*}

\definecolor{mygold}{HTML}{f7b801}
\definecolor{myred}{HTML}{f35b04}
\definecolor{myviolet}{HTML}{7678ed}
\begin{figure}\centering
\begin{tikzpicture}
\node [draw,minimum size=5cm,label=80:{$W^3$},text height=2cm] (W) at (0,0) {\hspace{3cm}$\langle u,v,w\rangle$}; 
\node [draw,ellipse,above left= 5mm of W.center,minimum size=1cm,text height=
5mm] (S) {$\langle a,b,c\rangle$}; 
\node [above right = 1mm of S.center] {$S$};
\node [below = of W.80] (R) {$R$};
\node [left = 16mm of W.330] {$-R$};
\node [left = of W.35,] {$\langle x,y,z\rangle$};
\draw (W.210) to [out=30,in=160] (W.0);

\node [draw,minimum size=5cm,label=80:{$\underline{W}^3$},right=2cm of W,text height=2cm] (Wp) {\hspace{3cm}$c(u,v,w)$}; 
\draw (Wp.30) -- (Wp.210);
\node [below = of Wp.80] {$\uR$};
\node [left = 16mm of Wp.330] {$-\uS$};
\node [ellipse,above left= 5mm of Wp.center,minimum size=1cm,text height=
5mm] {$c(a,b,c)$}; 
\node [left = of Wp.35] {$c(x,y,z)\setminus A$};
\node [left =3mm of Wp.15] {$A$};

\coordinate (a) at (Wp.210);
\coordinate (b) at (Wp.30);
\coordinate (c) at (Wp.0);
\coordinate (m) at ($(a)!.5!(b)$);

\draw [dotted] (m) to [out=30,in=160] (c);

\draw [-stealth,shorten <=5pt,shorten >=5pt,gray] (W.20) -- (Wp.160);
\end{tikzpicture}\caption{A visual presentation of the construction of the special frame of a frame $\langle W,R,S\rangle$ with $S\subseteq R$. Observe that $A\defeq\{\langle x,y,z\rangle,\langle x',y',z'\rangle\}$}\label{fig:special-model}
\end{figure}

\begin{lemma}\label{eqRS}
If $\frF$ is a wMIA frame, then $\uR = \uS$.   \end{lemma}
\begin{proof}
We are going to show that $\uR\cup-\uS=\uW^3$ and $\uR\cap-\uS=\emptyset$.

For the first equality, we have to prove that
\[
\uW^3=m(S)\cup[m(R\setminus S)\setminus(R\cup R')]\cup m(-R)\cup[(R\cup R')\setminus(S\cup S')]. 
\]
By \eqref{eq:m-unity} and \eqref{eq:m-union} we have
\[
\uW^3=m(W^3)=m(S) \cup m(R \setminus S) \cup m(-R),
\]
so it is enough to show that
\[
m(R \setminus S)=[m(R\setminus S)\setminus(R\cup R')]\cup[(R\cup R')\setminus(S\cup S')]. 
\]
Indeed, $S\cup S'\subseteq R\cup R'$ and $m(S)\cap(R\cup R')=S\cup S'$, while from  Lemma~\ref{lem:m-properties} it follows that $m(S)\subseteq m(R)$, $S\cup S'\subseteq m(S)$, $R\cup R'\subseteq m(R)$. Together these imply
\[
m(R) \setminus m(S)=[(m(R)\setminus m(S))\setminus(R\cup R')]\cup[(R\cup R')\setminus(S\cup S')].\footnotemark
\]
\footnotetext{See Figure~\ref{fig:mR-mS}.}
Using Lemma~\ref{lem:m-properties} again, we obtain $m(R\setminus S)=m(R) \setminus m(S)$.

For the second equality, observe that each of the following four sets is empty (again, see Figure~\ref{fig:mR-mS} for diagrammatic representation):
\begin{gather*}
 m(S)\cap m(-R),\\ 
 m(S)\cap [(R\cup R')\setminus(S\cup S')],\\
 [m(R\setminus S)\setminus(R\cup R')]\cap m(-R),\\
 [m(R\setminus S)\setminus(R\cup R')]\cap [(R\cup R')\setminus(S\cup S')].
\end{gather*}
The first one is empty by \eqref{eq:m-intersection}, since $S\cap-R=\emptyset$. The second one is, since $(R\cup R')\setminus(S\cup S')$ has only pure triples that are in $R\cup R'$, but not in $m(S)$. The emptiness of the third follows from \eqref{eq:m-complement}, and the fourth one has purely set-theoretical justification.

This concludes the proof.
\end{proof}

\begin{figure}
\begin{tikzpicture}
\node [draw,minimum size=5.5cm,label=80:{$\uW^3$},text height=2cm] (W) at (0,0) {}; 
\node [draw,ellipse,right = 5mm of W.150,minimum width = 2cm, minimum height = 1.7cm,text width=
20mm] (S) {\hspace{-1em}$m(S)$}; 
\node [right = 5mm of W.185] (R) {$m(R)$};
\node [left = 5mm of W.330] (-R) {$m(-R)$};
\node [minimum width=2.7cm,minimum height=2cm,draw,rounded corners=10,thick] (C) at (S.350) {};
\node [left = 5mm of W.15] (RR) {$R\cup R'$};
\node [left = 1mm of S.0] (SS) {$S\cup S'$};
\draw (W.210) to [out=30,in=160] (W.350);
\end{tikzpicture}\caption{$m(R)\setminus m(S)$ is equal to the union of $[m(R)\setminus m(S)]\setminus (R\cup R')$ and $(R\cup R')\setminus (S\cup S')$.}\label{fig:mR-mS}
\end{figure}

\pagebreak

\begin{lemma}\label{lem:for-modal-equivalence}
    If $\frF\defeq\klam{W,R,S}$ is a wMIA frame, then $\klam{r,s,t}\in\uR$ implies $\klam{j(r),j(s),j(t)}\in R$.
\end{lemma}
\begin{proof}
    Assume that $\klam{j(r),j(s),j(t)}\in-R$. By the relevant definitions, 
    \[
    c(j(r),j(s),j(t))\subseteq m(-R)\subseteq-\uS,
    \]so by Lemma \ref{eqRS} we have that $c(j(r),j(s),j(t))\subseteq-\uR$. Consequently, $\klam{r,s,t}$ is in the complement of $\uR$ as required.
\end{proof}
\begin{theorem}\label{th:modal-equivalence}
If $\frM \df \klam{\frF,v}$ is a model of \Kpm, based on a wMIA frame $\frF\defeq\klam{W,R,S}$, then there is a special model $\underline{\frM} \df \klam{\uW,\uR,\uS, \uv}$ such that $\frM$ and $\underline{\frM}$ are modally equivalent.
\end{theorem}
\begin{proof}
Fix a model $\frM\defeq\klam{\frF,v}$. Let $\frF'$ be an isomorphic and disjoint copy of $\frF$. Put $\frM'\defeq\klam{\frF',v'}$ where $v'(\varphi)\defeq\{w'\mid w\in v(\varphi)\}$. Let $\frF_s\defeq\klam{\uW,\uR,\uS}$ be the special frame obtained from $m(\frF)$.
Let $\uv(p)\defeq v(p)\cup v'(p)$ for every propositional variable $p$, and extend it to the valuation of all formulas $\uv\colon\Fml\to 2^{\uW}$ in the standard way. We will prove by induction on the complexity of the formulas that for any $\varphi\in\Fml$, $\uv(\varphi)=v(\varphi)\cup v'(\varphi)$. 

For the negation, we have
\begin{align*}
    \uv(\neg\varphi)\defeq\uW\setminus\uv(\varphi)&{}=\uW\setminus(v(\varphi)\cup v'(\varphi))\\
    &{}=(\uW\setminus v(\varphi))\cap (\uW\setminus v'(\varphi))\\
    &{}=[(W\setminus v(\varphi))\cup W']\cap[(W'\setminus v'(\varphi))\cup W]\\
    &{}=(W\setminus v(\varphi))\cup (W'\setminus v'(\varphi))\\
    &{}=v(\neg\varphi)\cup v'(\neg\varphi).
\end{align*}
Other Boolean connectives are straightforward, so we move to the modal operators.

Suppose $\uv(\varphi)=v(\varphi)\cup v'(\varphi)$ and $\uv(\psi)=v(\psi)\cup v'(\psi)$, we will prove that \[\uv(\Diamond (\varphi,\psi))=v(\Diamond (\varphi,\psi))\cup v'(\Diamond (\varphi,\psi))
\]
and 
\[
\uv(\win (\varphi,\psi))=v(\win (\varphi,\psi))\cup v'(\win (\varphi,\psi)).
\]

Let $s\in\uW$ be such that $s\in \uv(\Diamond (\varphi,\psi))$, i.e., $(\uv(\varphi)\times \{s\}\times \uv(\psi))\cap \uR\neq \emptyset$. By the induction hypothesis, we have that
\[
[(v(\varphi)\cup v'(\varphi))\times \{s\}\times (v(\psi)\cup v'(\psi))]\cap\uR\neq\emptyset.
\]
If $\klam{r,s,t}$ is in the intersection, then by Lemma~\ref{lem:for-modal-equivalence} the triple $\klam{j(r),j(s),j(t)}$ is in $R$. Clearly, $j(r)\in v(\varphi)$ and $j(t)\in v(\psi)$. So, if $s\in W$, then $j(s)=s$ and $\klam{j(r),s,j(t)}\in (v(\varphi)\times\{s\}\times v(\psi))\cap R$. If $s\in W'$, then $j(s)'=s$ and $\klam{j(r)',s,j(t)'}\in (v'(\varphi)\times\{s\}\times v'(\psi))\cap R'$. Thus $s\in v(\Diamond(\varphi,\psi))\cup v'(\Diamond(\varphi,\psi))$.


Let $s\in v(\Diamond(\varphi,\psi))\cup v'(\Diamond(\varphi,\psi))$, and consider the case that $s$ is in the first component of the union. Then $s\in W$ and $v(\varphi)\times \set{s}\times v(\psi)\cap R\neq \emptyset$. So, there exist $y\in v(\varphi)$ and $z\in v(\psi)$ such that $\klam{y,s,z}\in R$.
By definition, $\klam{y',s,z'}\in \uR$ and, by the induction hypothesis,  $(\uv(\varphi)\times \set{s}\times \uv(\psi))\cap \uR\neq \emptyset $. Therefore $s\in \uv(\Diamond (\varphi,\psi))$. The proof is analogous for $s\in v'(\Diamond(\varphi,\psi))$.

For the case of the sufficiency operator, let $s\in \uv(\win(\varphi,\psi))$, i.e., $\uv(\varphi)\times \{s\}\times \uv(\psi)\subseteq \uR$. If $s\in W$, let $y\in v(\varphi)$ and $z\in v(\psi)$. By the induction hypothesis, $\klam{y,s,z}\in \uR$. As the triple is pure, by the construction of $\uR$ it must be an element of $S$. Consequently, we obtain $v(\varphi)\times \{s\}\times v(\psi)\subseteq S$, i.e., $s\in v(\win(\varphi,\psi))$. If $s\in W'$, it follows by a similar reasoning that $s\in v'(\win(\varphi,\psi))$.

Let $s\in v(\win(\varphi,\psi))\cup v'(\win (\varphi,\psi))$. If $s\in v(\win(\varphi,\psi))$, then $v(\varphi)\times \set{s}\times v(\psi)\subseteq S$. Let $r\in \uv(\varphi)=v(\varphi)\cup v'(\varphi)$ and $t\in \uv(\psi)=v(\psi)\cup v'(\psi)$. So, $j(r)\in v(\varphi)$ and $j(t)\in v(\psi)$, and the assumption entails that $\klam{j(r),s,j(t)}\in S$. By the construction of $\uR$, we have $c(j(r),s,j(t))\subseteq\uR$, and as the triple $\klam{r,s,t}$ is in the cell, it must be in $\uR$, too.
The proof is analogous for $s\in v'(\win(\varphi,\psi))$.

Since $\frM$ and $\frM'$ are copies of each other, it is easy to see that the model $\underline{\frM}\df \klam{\uW,\uR,\uS,\uv}$ is modally equivalent to $\frM$ and by Lemma \ref{eqRS}, we have $\uR=\uS$.
\end{proof}

\section{The class \texorpdfstring{$\Eq(\wmia)$}{Eq(WMIA)}}\label{sec:kmia}

In this section we shall exhibit an axiom system for the equational class of algebraic models of \Kpm. To this end, consider the following equations for the operator $u$ introduced by \eqref{def:u} in Section~\ref{sec:notation-et-al}:
\begin{align}
& u(a,\zero) \leq a, \label{eq:u-4.1} \\
& u(a,\zero) \leq u(u(a,\zero),\zero), \label{eq:u-4.2} \\
& a \leq u(u^{\partial}(a,\zero),\zero),  \label{eq:u-4.3} \\
&u(a,a)=u(a,\zero)=u(\zero,a), \label{eq:u-4.4} \\
& u(a,b) = u(a,\zero) + u(\zero,b).\label{eq:u-4.5}
\end{align}
Recall that the class of wMIAs was defined as a subclass of the class of PS-algebras incorporating as an axiom the condition \eqref{wMIA}, which is not an identity. We consider the class $\Eq(\wmia)$, the smallest variety containing $\wmia$. 

Suppose that
\[
\Sigma\defeq\{\eqref{eq:f-normality},\ldots,
\eqref{eq:g-coadditivity},\eqref{eq:u-4.1},\ldots,\eqref{eq:u-4.5}\}
\]
and let $\Va(\Sigma)$ be the subvariety of PS-algebras generated by the equations \eqref{eq:u-4.1}--\eqref{eq:u-4.5}.

We will prove that the two varieties coincide: 

\begin{theorem}\label{thm:eqMIA}
$\Va(\Sigma) = \Eq(\wmia)$.
\end{theorem}

Along the way we shall exhibit a larger class of algebras, which is relevant to our considerations.

In the case of unary modalities investigated in \citep{dot_mixed} a unary PS-algebra $\klam{A,f,g}$ is a weak MIA \tiff the mapping defined by $u'(a) \df f^\partial(a) \cdot g(-a)$ is the dual of the unary discriminator. We have shown in \citep{dgm23} that for binary modalities such equivalence does not hold any more, and the weaker condition
\begin{equation}\tag{$\boldsymbol{\mathsf{di}}$}\label{eq:5}
(\forall a,b \in A)[a \cdot b \neq \zero \Rarrow g(a,b) \leq f(a,b)].\footnotemark
\end{equation}
\footnotetext{This is condition (5) from the above-mentioned paper.}
is necessary and sufficient for the discriminator to exist.
\begin{lemma}[{\citealp[Lemma 6.13]{dgm23}}]\label{oldlemma}
Suppose that $\frA \df \klam{A,f,g}$ is a PS-algebra. Then, $d(a) \df f(a,a) + -g(a,a)$ is the unary discriminator \tiff $\frA$ satisfies  \eqref{eq:5}.  
\end{lemma}

Clearly, every wMIA satisfies \eqref{eq:5}. On the other hand, Table 1 from \citep{dgm23} shows that \eqref{wMIA} is not equivalent to \eqref{eq:5}.

Let $\frA \df \klam{A,f,g}$ be a PS-algebra. We say that $\frA$ is a \emph{dMIA} if it satisfies \eqref{eq:5}.  
The class of all dMIAs will be denoted by $\dmia$.

\begin{lemma}\label{lem:udisc}
   Let $\frA\df \klam{A,f,g}$ be a dMIA. Then, $\frA$ satisfies \eqref{eq:u-4.4} \textup{(}see Figure~\ref{fig:intuition-for-lemma-h001}\textup{)}.
\end{lemma}
 \begin{proof}
     Firstly, assume that $a=\one$. Then,
     \[
        u(a,\zero)=u(\one,\zero)=f^\partial(\one,\zero)\cdot g(\zero,\one)=\one\cdot\one=\one.
     \]
     By analogous reasoning $u(\zero,a)=\one$ and $u(a,a)=\one$.
     
     Secondly, let $a\neq\one$, i.e., $-a\cdot\one\neq\zero$. Then by \eqref{eq:5} we have $g(-a,\one)\leq f(-a,\one)$, and thus
     \[
        u(a,\zero)=-f(-a,\one)\cdot g(-a,\one)=\zero
     \]
    In a similar way we show that $u(\zero,a)=\zero$ and $u(a,a)=\zero$. 
\end{proof}

\begin{figure}
\begin{tikzpicture}
    \coordinate (00) at (0,0) coordinate (10) at (3,0) coordinate (01) at (0,3) coordinate (11) at (3,3) coordinate (0) at (6,0) coordinate (1) at (6,3) coordinate (d) at ($(00)!0.57!(11)$) coordinate (sh) at ($(10)!0.5!(11)$) coordinate (eh) at ($(0)!0.5!(1)$) coordinate (dMIA) at ($(00)!0.5!(01)$);
    \draw (00) rectangle (11);
    \draw [shorten <=8pt,shorten >=8pt,-stealth] (sh)  -- (eh);
  \draw (0) -- (1) node [above] at ($(sh)!0.5!(eh)$) {$u$};
    \node [below left] at (00) {$\klam{\zero,\zero}$} 
        node [above left] at (01) {$\klam{\zero,\one}$}
        node [above right] at (11) {$\klam{\one,\one}$}
        node [below right] at (10) {$\klam{\one,\zero}$}
        node [above right] at (1) {$\one$}
        node [below right] at (0) {$\zero$};
    \draw [gray,ultra thick]  (00) -- (10)  (00) -- (01)  (00) -- (11);
    \node (p00) [circle,fill=gray,inner sep=2pt]  at (00) {}
    node (p10) [circle,fill=black,inner sep=2pt]  at (10) {}
    node (p01) [circle,fill=black,inner sep=2pt]  at (01) {}
    node (p11) [circle,fill=black,inner sep=2pt] at (11) {}
    node (p0) [circle,fill=gray,inner sep=2pt] at (0) {}
    node (p1) [circle,fill=black,inner sep=2pt] at (1) {}
    node [left=4pt] (delta) at (d) {$\Delta$}; 
    \node [anchor=east,outer sep=60pt] at (dMIA) {$\mathbf{dMIA}$};
\end{tikzpicture}\caption{Assume that $\langle A,\mathord{\leq_\ell}\rangle$ is a linearization (i.e., $\leq_\ell$ is an extension of $\leq$ which is a linear order compatible with $\leq$) of a dMIA $A$, and take its binary product. Then, according to lemmas  \ref{oldlemma} and~\ref{lem:udisc}, the binary operator $u$ always sends three black pairs to $\one$, and the gray ones to $\zero$.}\label{fig:intuition-for-lemma-h001}
\end{figure}
 
\begin{lemma}\label{equationwMIA}
 Suppose that $\frA \df \klam{A,f,g}$ is a weak MIA. Then, $\frA$ satisfies \eqref{eq:u-4.5}, and the range of the $u$ operator is $\{\zero,\one\}$ \textup{(}see Figure~\ref{fig:intuition-for-lemma-equationwMIA}\/\textup{)}. 
 \end{lemma}
\begin{proof}
     If $a=\one$ or $b=\one$, then by the mere fact that $u$ is a necessity operator we obtain $u(a,b)=u(a,\zero)+u(\zero,b)$.

If $a\neq \one$ and $b\neq \one$, then $-a\neq \zero$ and $-b\neq \zero$, and by \eqref{wMIA}, $g(-a,-b)\leq f(-a,-b)$. Therefore,
\[u(a,b)=g(-a,-b)\cdot -f(-a,-b)=\zero.\]
In particular $u(a,a)=\zero$ and $u(b,b)=\zero$.
By Lemma \ref{lem:udisc}, $u(a,\zero)=u(a,a)$ and $u(\zero,b)=u(b,b)$, so the claim follows.
\end{proof}

\begin{figure}
\begin{tikzpicture}
    \coordinate (00) at (0,0) coordinate (10) at (3,0) coordinate (01) at (0,3) coordinate (11) at (3,3) coordinate (0) at (6,0) coordinate (1) at (6,3) coordinate (d) at ($(00)!0.57!(11)$) coordinate (sh) at ($(10)!0.5!(11)$) coordinate (eh) at ($(0)!0.5!(1)$) coordinate (wMIA) at ($(00)!0.5!(01)$);
    \fill [gray,opacity=0.3] (00) rectangle (11);
    \draw [shorten <=8pt,shorten >=8pt,-stealth] (sh)  -- (eh);
  \draw (0) -- (1) node [above] at ($(sh)!0.5!(eh)$) {$u$};
    \node [below left] at (00) {$\klam{\zero,\zero}$} 
        node [above left] at (01) {$\klam{\zero,\one}$}
        node [above right] at (11) {$\klam{\one,\one}$}
        node [below right] at (10) {$\klam{\one,\zero}$}
        node [above right] at (1) {$\one$}
        node [below right] at (0) {$\zero$};
    \draw [gray,ultra thick]  (00) -- (10)  (00) -- (01)  (00) -- (11);
    \node (p00) [circle,fill=gray,inner sep=2pt]  at (00) {}
    node (p10) [circle,fill=black,inner sep=2pt]  at (10) {}
    node (p01) [circle,fill=black,inner sep=2pt]  at (01) {}
    node (p11) [circle,fill=black,inner sep=2pt] at (11) {}
    node (p0) [circle,fill=gray,inner sep=2pt] at (0) {}
    node (p1) [circle,fill=black,inner sep=2pt] at (1) {}
    node [left=4pt] (delta) at (d) {$\Delta$}; 
    \node [anchor=east,outer sep=60pt] at (wMIA) {$\mathbf{wMIA}$};
\end{tikzpicture}\caption{This time let $\langle A,\mathord{\leq_\ell}\rangle$ be a linearization of a wMIA $A$, and consider its binary product. Then, according to Lemma~\ref{equationwMIA}, the binary operator $u$ sends the three black pairs to $\one$, and all the remaining ones to $\zero$. The less saturated fragment of the rectangle indicates the area of possible differences between the behavior of $u$ for dMIAs and wMIAs.}\label{fig:intuition-for-lemma-equationwMIA}
\end{figure}

\begin{figure}
\centering
\[
 \begin{array}{c|cccccc|cccc}
f & \zero & a & b & \one &&g& \zero & a & b &\one    \\ \hline
\zero & \zero & \zero &\zero &\zero  && \zero & \one & \one & \one & \one \\
a & \zero & \zero & \zero & \zero && a & \one & b & b & b\\
b &\zero &\zero & \one &\one && b &\one &b &\zero &\zero\\
\one & \zero &\zero &\one &\one && \one &\one &b &\zero &\zero
 \end{array}
 \]\caption{A PS-algebra that satisfies \eqref{eq:u-4.5} but is not a weak MIA.}\label{fig:not-wMIA}
\end{figure}

The algebra in Figure~\ref{fig:not-wMIA} is an example of a PS-algebra that satisfies \eqref{eq:u-4.5} but not \eqref{wMIA}, which shows that the condition is strictly weaker than \eqref{wMIA}. 

\begin{proposition}\label{prop:wMIA-iff-2-conditions}
     Let $\frA\df \klam{A,f,g}$ be a PS-algebra. Then, $\frA$ is a wMIA \tiff it satisfies both \eqref{eq:u-4.5} and  \eqref{eq:5}.
 \end{proposition}
  \begin{proof}
     Suppose that $\frA$ satisfies \eqref{eq:u-4.5} and \eqref{eq:5}. Observe that 
     \[
        \text{if}\ x\neq \one,\ \text{then}\ u(x,\zero)=\zero.
     \]
     Indeed, if $x\neq \one$, then $-x\neq \zero$ and by \eqref{eq:5} we have
     $g(-x,\one)\leq f(-x,\one)$, and so: $u(x,\zero)=g(-x,\one)\cdot -f(-x,\one)=\zero$, as required. 
     
     To show that wMIA holds, assume that $a\neq \zero$ and $b\neq \zero$, i.e., $-a\neq \one$ and $-b\neq \one$. By the above we have $u(-a,\zero)=\zero$ and $u(-b,\zero)=\zero$. By \eqref{eq:u-4.5} and \eqref{eq:u-4.4} (which is a consequence of \eqref{eq:5} by  Lemma~\ref{lem:udisc}) $u(-a,-b)=\zero+\zero=\zero$. Thus, $-f(a,b)\cdot g(a,b)=\zero$, i.e., $g(a,b)\leq f(a,b)$. Therefore, $\frA$ is a weak MIA.

     The other direction is immediate by Lemma \ref{equationwMIA}.
 \end{proof}
 
 Using again the algebra in Figure~\ref{fig:not-wMIA} we can see in light of Proposition~\ref{prop:wMIA-iff-2-conditions} that \eqref{eq:u-4.5} does not entail \eqref{eq:5}. Also, the algebra from Table 1 of \citep{dgm23} shows that \eqref{eq:5} is strictly weaker than \eqref{wMIA}, and thus \eqref{eq:5} cannot entail \eqref{eq:u-4.5} either. Therefore, we obtain
 \begin{corollary}
     \eqref{eq:u-4.5} and \eqref{eq:5} are independent in the class of PS-algebras.
 \end{corollary}

 \begin{lemma}\label{propertiesu}
     Suppose that $\frA \df \klam{A,f,g}$ is a weak MIA. Then, $\frA$ satisfies \eqref{eq:u-4.1}--\eqref{eq:u-4.3}, and for all $a,b\in A$,  $u(a,b)=u(b,a)$.
\end{lemma}
 \begin{proof}
     For commutativity, by \eqref{eq:u-4.5} and \eqref{eq:u-4.4}, $u(a,b)=u(a,\zero)+ u(\zero,b)=u(b,\zero)+ u(\zero,a)=u(b,a)$.

     \eqref{eq:u-4.1} If $a\neq\one$, then by Lemma \ref{lem:udisc}, $u(a,\zero)=\zero\leq a$.
     
     \eqref{eq:u-4.2} If $a=\one$, then $u(\one,\zero)=\one=u(u(\one,\zero),\zero)$. If $a\neq \one$ then $u(a,\zero)=\zero$ and the result follows.

     \eqref{eq:u-4.3} If $a=\one$, $u^\partial(\one,\one)=-u(\zero,\zero)=-\zero=\one$. So, $\one\leq u(\one,\zero)=\one$.
     If $\zero<a<\one$, $u^\partial(a,\one)=-u(-a,\zero)=-\zero=\one$. Thus, $a\leq u(\one,\zero)=\one$.
 \end{proof}
 \begin{remark}
Note that Condition \eqref{eq:u-4.4} contains two identities, $u(a,\zero)=u(\zero,a)$ and $u(a,a)=u(\zero,a)$. We need the second identity to prove that simple algebras in the variety generated by \eqref{eq:u-4.1}--\eqref{eq:u-4.4} are exactly the elements of the class $\dmia$. In Figure \ref{fig:ind4.4}, there is an example of a PS-algebra in the variety generated by \eqref{eq:u-4.1}--\eqref{eq:u-4.3} that satisfies the first but not the second identity.
\end{remark}
 
 \begin{figure}
\centering
\[
 \begin{array}{c|cccccc|cccccc|cccc}
f & \zero & a & b & \one &&g& \zero & a & b &\one &&u& \zero & a & b &\one    \\ \hline
\zero & \zero & \zero &\zero &\zero  && \zero & \one & \one & \one & \one && \zero & \zero & \zero & \zero & \one \\
a & \zero & \zero & \one & \one && a & \one & \zero & \zero & \zero && a & \zero & \one & \zero & \one\\
b &\zero &\zero & \zero &\zero && b &\one &\zero &\one &\zero && b & \zero & \zero & \zero & \one\\
\one & \zero &\zero &\zero &\one && \one &\one &\zero &\zero &\zero && \one & \one & \one & \one & \one
 \end{array}
 \]\caption{A PS-algebra that satisfies \eqref{eq:u-4.1}--\eqref{eq:u-4.3} and the equation $u(c,\zero)=u(\zero,c)$ for all $c$, but $u(a,a)=\one\neq u(a,\zero)=\zero$.}\label{fig:ind4.4}
\end{figure}

Let $u^1\colon A \to A$ be the unary necessity operator defined by:
\begin{equation*}
    u^1(a)=u(a,\zero).
\end{equation*}
It is easy to see that the dual ${u^1}^\partial$ is the possibility operator defined by 
\begin{equation*}
{u^1}^\partial(a)=u^\partial(a,\one). 
\end{equation*}
Indeed, $-{u^1}(-a)=-u(-a,\zero)=u^\partial(a,\one)={u^1}^\partial(a)$.

So, by Lemma~\ref{propertiesu}, we obtain that $u^1$ is the standard S5 operator. 

\begin{lemma}
Let $\frA\defeq\klam{A,f,g}$ be a PS-algebra that satisfies \eqref{eq:u-4.2} and \eqref{eq:u-4.4}. If $a\in A$, then $\ua{u(a,\zero)}$ is a congruence filter on $\frA$. 
\end{lemma}
\begin{proof}
    If $u(a,\zero)\leq x$, then $u(u(a,\zero),\zero)\leq u(x,\zero)$. By \eqref{eq:u-4.2} we have that $u(a,\zero)\leq u(x,\zero)$. By \eqref{eq:u-4.4} we have that $u(x,\zero)\cdot u(\zero,x)\in \ua{u(a,\zero)}$, and so Theorem~\ref{th:congruence-filter-via-u} entails that $\ua{u(a,\zero)}$ is a congruence filter.
\end{proof}

\begin{lemma}\label{semisimplicity}
Suppose that $\frA \df \klam{A,f,g}$ is a subdirectly irreducible PS-algebra in the variety generated by \eqref{eq:u-4.1}--\eqref{eq:u-4.4}. Then $\frA$ is simple, and $u^1(a) \df u(a,\zero)$ is the dual discriminator.
\end{lemma}
\begin{proof}
We first show that $\frA$ is simple. Since $\frA$ is subdirectly irreducible, there is the smallest non-identity congruence $\theta$. In consequence, $F_\theta$ is a non-trivial (i.e., ${}\neq\{\one\}$) congruence filter, which is the smallest among the non-trivial congruence filters of~$A$. From the latter, we obtain 
that $F_{\theta}$ is principal, say $F_{\theta} = \ua{c}$. Indeed, let $x\in F_{\theta}\setminus\{\one\}$. Then $\ua{u(x,\zero)}\subseteq F_{\theta}$, since $F_{\theta}$ is a congruence filter, and $F_{\theta}\subseteq\ua{u(x,\zero)}$ by the lemma above and the fact that $F_{\theta}$ is the smallest non-trivial congruence filter on $\frA$. 

Clearly, $c\neq \one$. Furthermore, $c\leq u(c,\zero)$, since $u(c,\zero)\in F_{\theta}$, and so 
\begin{equation}\tag{$\dagger$}
c=u(c,\zero) 
\end{equation}
by \eqref{eq:u-4.1}. Also, 
\begin{equation}\tag{$\ddagger$}
u(a,\zero) \leq c\quad \text{for all}\ a \neq \one.
\end{equation} 
For this, observe that in case $a\neq\one$, $\ua{u(a,\zero)}$ is a congruence filter for which $F_{\theta}\subseteq\ua{u(a,\zero)}$, and so $c\in \ua{u(a,\zero)}$.  

Assume $c \neq \zero$. Then,
\begin{xalignat*}{2}
-c &\leq u(-u(c,\zero),\zero) &&\text{by \eqref{eq:u-4.3}} \\
&= u(-c,\zero) &&\text{by ($\dagger$)}\\
&\leq c &&\text{by ($\ddagger$), since }-c\neq \one.
\end{xalignat*}
This contradicts $c \neq\one$. Thus, $F_{\theta}=A$ and $\theta=A\times A$, which means that $\frA$ is simple. 

Simplicity of $\frA$ implies that
\begin{gather}
u(a,\zero) = 
\begin{cases}
\one, &\text{if } a = \one, \\
\zero, &\text{otherwise}.
\end{cases}
\end{gather}
Indeed, if $a=\one$, then $u(a,\zero)=\one$, since $u$ is a necessity operator. On the other hand, if  $a\neq\one$, then $u(a,\zero)\neq\one$ by \eqref{eq:u-4.1}. So, since for the congruence filter $F\defeq\ua{u(a,\zero)}$, $\theta_F$ is either the identity or the universal relation, it must be the case that $\theta_F=A\times A$, and using Lemma~\ref{lem:Bcong} we obtain $F_{\theta_{F}}=F=A$. Thus, $u(a,\zero)=\zero$.

This completes the proof.
\end{proof}

 \begin{lemma}\label{lem:equational}
     If $\frA$ is a subdirectly irreducible PS-algebra in the variety generated by \eqref{eq:u-4.1}--\eqref{eq:u-4.4}, then \eqref{eq:5} holds in $\frA$. In particular, if $\frA$ is a subdirectly irreducible PS-algebra in the variety generated by \eqref{eq:u-4.1}--\eqref{eq:u-4.5}, then $\frA\in \wmia$.
 \end{lemma}
 \begin{proof}
   Let $\frA$ be a subdirectly irreducible PS-algebra in the above-mentioned variety. By Lemma \ref{semisimplicity} and \eqref{eq:u-4.4}, 
\begin{gather}
u(a,a) = 
\begin{cases}
\one, &\text{if } a = \one, \\
\zero, &\text{otherwise}.
\end{cases}
\end{gather}
Then, $d(a)=-u(-a,-a)=f(a,a)+-g(a,a)$ is the unary discriminator. From Lemma~\ref{oldlemma}, we obtain \eqref{eq:5} holds in $\frA$. If in addition $\frA$ satisfies \eqref{eq:u-4.5}, then by Proposition \ref{prop:wMIA-iff-2-conditions}, $\frA\in \wmia$.
 \end{proof}

We can now prove Theorem \ref{thm:eqMIA}:

\begin{proof}
($\subseteq$) From Lemma \ref{lem:equational} we know that every subdirectly irreducible $\frA$ in the variety $\Eq(\Sigma)$ is a wMIA. In consequence  $\Eq(\Sigma)\subseteq\Eq(\wmia)$.

($\supseteq$) \eqref{eq:u-4.1}--\eqref{eq:u-4.5}  hold in \wmia by Lemmas \ref{propertiesu} and \ref{lem:udisc}. So if $\frA$ is in $\Eq(\wmia)$, then $\frA$ must satisfy these equations too.
\end{proof}

\begin{figure}[htb]\label{fig:vars}\centering
 \begin{tikzpicture}
      \draw (0.5,1.5) rectangle (5.5,4.25);
     \node at (3,3) {\begin{minipage}{.7\textwidth}
     \begin{align*}
  & u(a,\zero) \leq a\\
 & u(a,\zero) \leq u(u(a,\zero),\zero) \\
 & a \leq u(-u(-a,\zero),\zero)\\
 & u(a,a) = u(a,\zero)= u(\zero,a)
 \end{align*}
 \end{minipage}};
 \draw (0.5,0.5) rectangle (5.5,1.25);
 \node at (3,0.875) {$u(a,b) = u(a,\zero) + u(\zero,b)$};
 \node at (6.75,2.75) {$\Eq(\dmia)$};
 \node at (6,0.875) {\eqref{eq:u-4.5}};
  \draw (-0.5,0.25) rectangle (8,4.5);
 \node at (9.25,2.25) {$\Eq(\wmia)$};
 \end{tikzpicture}
 \end{figure}

Finally in this section we will show an analogon to \cite[Theorem 8.5]{dot_mixed}. Let $T(\Var)$ be the term algebra over the language of \Kpm\ with the set $\Var$ of variables. Then, each formula $\phi(p_1,\ldots,p_n)$ of \Kpm\ can be regarded as an element of $T(\Var)$ via the natural translation $\tau$. If $\frA$ is an algebra with the signature of \Kpm, we say that \emph{$\phi(p_1,\ldots,p_n)$ is valid in $\frA$}---written as $\frA \models \phi(p_1,\ldots,p_n)$---if $\tau(\phi(v(p_1), \ldots, v(p_n))) = \one$ for all valuations $v: \Var \to A$. 

\begin{lemma}
Let $\frM\defeq\klam{W,R,S,v}$ be a model of~\Kpm, and $B_v = \set{v(\phi)\mid \phi\in \Fml}$. Then
\begin{enumerate}
\item $\B_v = \klam{B_v, \cap, \cup,\, \emptyset, X, \poss{R}, \suff{S}}$ satisfies \eqref{wMIA}.
\item If $\B$ is a subalgebra of $\Cm(\frM)$ and $v$ is a mapping onto a set of generators of $\B$, then $\B = \B_v$.
\end{enumerate}
\end{lemma}
\begin{proof}
1. By definition, the extension of $v$ over $T(\Var)$ is a homomorphism $T(\Var) \to \Cm(\frM)$, thus, $\B_v$ is a subalgebra of $\Cm(\frM)$. Since $\Cm(\frM)$ satisfies \eqref{wMIA}, and this is a universal sentence, we obtain that $\B_v$ satisfies \eqref{wMIA}.

2. This follows from the definition of the extension of $v$ and the fact that $v$ maps $\Var$ onto a set of generators.
\end{proof}

\section{3-frames for weak MIAs}

By a 3-frame, we mean any Kripke frame $\frF\defeq\klam{W,R}$ in which $R$ is a ternary relation on $W$. In \citep{dgm23} we asked the following question: \emph{is there for each betweenness algebra $\frA$ a 3-frame $\frF$ such that $\frA$ is embeddable into $\Cm(\frF)$?} Below, we answer this question positively by showing a stronger statement: any weak mixed algebra $\frA$ can be embedded into the complex algebra of a 3-frame, which is the special frame of the canonical frame of $\frA$. To this end we follow the construction developed in Section~\ref{sub:special-model}

\begin{figure}\centering
\begin{tikzpicture}
\node [draw,minimum size=6cm,label=80:{$\Ult A^3$},text height=2cm] (W) at (0,0) {\hspace{3.5cm}$\langle \ultV_1,\ultV_2,\ultV_3\rangle$};
\node [draw,ellipse,left= of W.center,minimum size=1.5cm,text height=
5mm,anchor=south] (S) {$\langle \ultW_1,\ultW_2,\ultW_3\rangle$};
\node [above right = 1mm of S.center] {$S_g$};
\node [below = 1.5cm of W.60] {$Q_f$};
\node [left = 16mm of W.330] {$-Q_f$};
\node [left = of W.35] {$\langle \ult_1,\ult_2,\ult_3\rangle$};
\draw (W.210) to [out=30,in=160] (W.0);

\node [draw,minimum size=6cm,label=80:{$\uUltA^3$},right=2cm of W,text height=2cm] (Wp) {\hspace{3cm}$c(\ultV_1,\ultV_2,\ultV_3)$};
\draw (Wp.30) -- (Wp.210);
\node [below = 1.5cm of Wp.70] {$\uR$};
\node [left = 16mm of Wp.330] {$-\uR$};
\node [ellipse,left= of Wp.center,minimum size=1cm,text height=
5mm,anchor=south] {$c(\ultW_1,\ultW_2,\ultW_3)$};
\node [left = of Wp.35] {$c(\ult_1,\ult_2,\ult_3)\setminus M$};
\node [left =3mm of Wp.15] {$M$};

\coordinate (a) at (Wp.210);
\coordinate (b) at (Wp.30);
\coordinate (c) at (Wp.0);
\coordinate (m) at ($(a)!.5!(b)$);

\draw [dotted] (m) to [out=30,in=160] (c);

\draw [-stealth,shorten <=5pt,shorten >=5pt,gray] (W.20) -- (Wp.160);
\end{tikzpicture}\caption{An application of the construction of the special frame to the canonical frame $\Cf^{ps}(\frA)=\langle\Ult A,Q_f,S_g\rangle$ of the wMIA $\frA$. $M\defeq\{\langle \ult_1,\ult_2,\ult_3\rangle,\langle \ult_1',\ult_2',\ult_3'\rangle\}$, where $\langle \ult_1,\ult_2,\ult_3\rangle\in Q_f\setminus S_g$.}\label{fig:special-model-for-Cm}
\end{figure}

\begin{theorem}\label{th:exists-3-frame}
    Let $\frA\df\klam{A,f,g}$ be a $\wmia$. Then, there exists a 3-frame $\frF$ such that $\frA$ is isomorphic to a subalgebra of $\Cm(\frF)$.
\end{theorem}
\begin{proof}
    Let $\Cf(\frA)=\klam{\Ult A, Q_f, S_g}$ be the canonical frame of $\frA$. Since $\Cf(\frA)$ is a wMIA frame, we construct its special frame following the method from subsection \ref{sub:special-model} (see Figure~\ref{fig:special-model-for-Cm}). To this end, let us consider an isomorphic copy $\Cf'(\frA)=\klam{\Ult A', Q'_f, S'_g}$ of $\Cf(\frA)$ where $\Ult A$ and $\Ult A'$ are disjoint sets. We put $\uUltA\df\Ult A\cup\Ult A'$, and define $\uR\subseteq \uUltA^3$ and its complement as before:

\begin{align*}
\uR \df{}& m(S_{g})\cup[(m(Q_{f})\cap m(-S_{g}))\setminus(Q_{f}\cup Q_{f}')]\\
={}&m(S_{g})\cup[m(Q_f\setminus S_g)\setminus(Q_{f}\cup Q_{f}')]\\
-\uR \df{}& m(-Q_{f})\cup[(Q_{f}\cup Q_{f}')\cap(-S_{g}\cup(-S_{g})')]\\
={}&m(-Q_f)\cup[(Q_{f}\cup Q_{f}')\setminus(-S_{g}\cup(-S_{g})')].
\end{align*}

Fix $\frF\df\klam{\uUltA,\uR}$. We will prove that $\frA$ embeds into the complex algebra $\Cm(\frF)$ via the mapping
$s\colon A\to 2^{\uUltA}$ defined by
    \[
s(a)\defeq\set{\ult\in\Ult A:a\in\ult}\cup\set{\ult'\in\Ult A':a\in\ult}.
    \]
First, note that $s(\zero)=\emptyset$ and $s(\one)=\uUltA$. It is easy to check that $s(a\cdot b)=s(a)\cap s(b)$, $s(a+ b)=s(a)\cup s(b)$ and $s(-a)=\uUltA\setminus s(a)$. We will prove that $s(f(a,b))=\poss{\uR}(s(a),s(b))$ and $s(g(a,b))=\suff{\uR}(s(a),s(b))$.

Let $\ult^{\epsilon}$ denote either the ultrafilter $\ult$ or its copy $\ult'$. Assume that $\ult^\epsilon\in s(f(a,b))$. By the definition, we have that $f(a,b)\in \ult$. Thus, there exist $\ult_a,\ult_b\in\Ult A$ such that $a\in \ult_a$, $b\in \ult_b$ and $f[\ult_a\times\ult_b]\subseteq \ult$,\footnote{See, e.g., \citep{jt51}.} i.e., $\klam{\ult_a,\ult,\ult_b}\in Q_f$. By construction, $\klam{\ult'_a,\ult^\epsilon,\ult_b}\in \uR$. It follows that $(s(a)\times\set{\ult^\epsilon}\times s(b))\cap \uR\neq \emptyset$ and therefore $\ult^\epsilon\in \poss{\uR}(s(a),s(b))$.

For the other direction, let $\ult^\epsilon\in \poss{\uR}(s(a),s(b))$, i.e., $(s(a)\times\set{\ult^\epsilon}\times s(b))\cap \uR\neq \emptyset$. Then, there exist $\ult_a,\ult_b\in \Ult A$ such that $\ult^\epsilon_a\in s(a)$, $\ult^\epsilon_b\in s(b)$ and $\klam{\ult^\epsilon_a,\ult^\epsilon,\ult^\epsilon_b}\in \uR$. By construction, $\klam{\ult_a,\ult,\ult_b}\in S_g$ or $\klam{\ult_a,\ult,\ult_b}\in Q_f\cap -S_g$. In both cases, $\klam{\ult_a,\ult,\ult_b}\in Q_f$, which implies that $f[\ult_a\times\ult_b]\subseteq\ult$ and therefore $f(a,b)\in \ult$. By definition, we obtain $\ult^\epsilon\in s(f(a,b))$.

Let $\ult^\epsilon\in s(g(a,b))$, which is to say that $g(a,b)\in \ult$. To prove that $s(a)\times \set{\ult^\epsilon}\times s(b)\subseteq \uR$, let us consider $\ult^\epsilon_a\in s(a)$ and $\ult^\epsilon_b\in s(b)$. By definition, $a\in \ult_a$ and $b\in \ult_b$ and we have $g[\ult_a\times\ult_b]\cap \ult\neq \emptyset$, i.e., $\klam{\ult_a,\ult,\ult_b}\in S_g$. By construction, $c(\ult_a,\ult,\ult_b)\subseteq \uR$, so $\klam{\ult^\epsilon_a,\ult^\epsilon,\ult^\epsilon_b}\in \uR$. It follows that $\ult^\epsilon\in \suff{\uR}(s(a),s(b))$.

For the reverse inclusion, assume that $\ult^\epsilon\in \suff{\uR}(s(a),s(b))$, that is $s(a)\times\set{\ult^\epsilon}\times s(b)\subseteq\uR$. Let $\ult_a,\ult_b\in \Ult A$ be such that $a\in\ult_a$ and $b\in\ult_b$. Then, $\klam{\ult_a,\ult^\epsilon,\ult_b}\in\uR$. Suppose that $\klam{\ult_a,\ult,\ult_b}\notin S_g$. By construction we have that $\klam{\ult_a,\ult,\ult_b}\in Q_f$ but then $\klam{\ult_a,\ult,\ult_b}\notin \uR$ and $\klam{\ult'_a,\ult',\ult'_b}\notin \uR$. This contradicts the fact that $s(a)\times\set{\ult^\epsilon}\times s(b)\subseteq\uR$, and thus, $\klam{\ult_a,\ult,\ult_b}\in S_g$. Since $\ult_a,\ult_b$ are arbitrary ultrafilters that contain $a$ and $b$, respectively, we obtain $g(a,b)\in \ult$. We conclude that $\ult^\epsilon\in s(g(a,b))$. 

Therefore $s$ is a homomorphism of Boolean algebras. Injectivity follows immediately from the properties of the Stone mapping. 
\end{proof}

The following corollary is a counterpart of \cite[Theorem 8.5]{dot_mixed}.

\begin{theorem}\label{thm:8.5}
Suppose that $\frA \df \klam{A,f,g} \in \wmia$. Then, there is a 3-frame $\frF \df \klam{W,R}$ such that $\frA$  and a subalgebra of $\Cm(\frF)$ satisfy the same equations.
\end{theorem}

\section{Applications to b-algebras}

We will show now how the results obtained in this paper can be applied to betweenness algebras introduced and studied in \cite{dgm23}. We begin by recalling the key concepts from the paper.

\begin{definition}
Let $\langle W,B\rangle$ be a 3-frame. $\Betw$ is called a \emph{betweenness relation} if it satisfies the following (universal) axioms:
\begin{gather*}
\Betw(a,a,a),\tag{BT0}\label{BT0}\\
\Betw(a,b,c)\rarrow\Betw(c,b,a),\tag{BT1}\label{BT1}\\
\Betw(a,b,c)\rarrow\Betw(a,a,b),\tag{BT2}\label{BT2}\\
\Betw(a,b,c)\wedge\Betw(a,c,b)\rarrow b=c\,\tag{BT3}\label{BT3}.
\end{gather*}
A 3-frame $\klam{W,B}$ satisfying these axioms is called a \emph{betweenness frame} (or just a \emph{b-frame}).
\end{definition}

\begin{definition}
    A weak MIA $\frA\defeq\langle A,f,g\rangle$ is a \emph{betweenness algebra} (\emph{b-algebra} for short) if $\frA$ satisfies the following axioms:
    \begin{gather}
        x\leq f(x,x),\tag{ABT0}\label{ABT0}\\
        f(x,y)\leq f(y,x),\tag{ABT1$_f$}\label{ABT1}\\
         g(x,y) \leq  g(y,x),\tag{ABT1$_g$}\label{ABT1g}\\
        y\cdot f(x,z)\leq f(x\cdot f(x,y),z),\tag{ABT2}\label{ABT2}\\
        f(x,g(x,-y)\cdot y)\leq y.\tag{ABT3}\label{ABT3}
     \end{gather}
\end{definition}
Assuming that $\Abtw$ is the class of betweenness algebras and putting
\[
\Sigma^B\defeq\Sigma\cup\{\eqref{ABT0},\ldots,\eqref{ABT3}\},
\]
from Theorem~\ref{thm:eqMIA} we obtain
\begin{corollary}
    $\Eq\left(\Sigma^B\right)=\Eq(\Abtw)$.
\end{corollary}

In \cite[Example 8.4]{dgm23} we demonstrated that there is a b-algebra that cannot be embedded into the complex algebra of any b-frame. However, as a consequence of Theorem~\ref{th:exists-3-frame} we have the following

\begin{corollary}
    Every b-algebra is isomorphic to a subalgebra of the complex algebra of a 3-frame.
\end{corollary}

As mentioned before, this answers problem (1) from \cite[Section 9]{dgm23}. We know from Example 8.4 of the paper that the ternary relation $\uR$ from Theorem~\ref{th:exists-3-frame} cannot be a betweenness relation, yet we may ask about its general properties. On the positive side, we have that $\uR$ satisfies \eqref{BT1}
\begin{equation*}
    \uR(\ultVeps_1,\ultVeps_2,\ultVeps_3)\rarrow \uR(\ultVeps_3,\ultVeps_2,\ultVeps_1)
\end{equation*}
\begin{proof}
    Indeed, assume that $\uR(\ultVeps_1,\ultVeps_2,\ultVeps_3)$. By construction, we have that either there is a triple $\klam{\ult_1,\ult_2,\ult_3}\in S_g$ such that
    \[
    \klam{\ultVeps_1,\ultVeps_2,\ultVeps_3}\in c(\ult_1,\ult_2,\ult_3)
    \]
    or there is a triple $\klam{\ult_1,\ult_2,\ult_3}\in Q_f\cap-S_g$ such that
    \[
        \klam{\ultVeps_1,\ultVeps_2,\ultVeps_3}\in c(\ult_1,\ult_2,\ult_3)\setminus\{\klam{\ult_1,\ult_2,\ult_3},\klam{\ult_1',\ult_2',\ult_3'}\}.
    \]
    In the first case, by \cite[Lemma 7.2]{dgm23} we obtain that the triple $\klam{\ult_3,\ult_2,\ult_1}$ is in $S_g$, so \[
    \klam{\ultVeps_3,\ultVeps_2,\ultVeps_1}\in c(\ult_3,\ult_2,\ult_1)\subseteq\uR.
    \]
    In the second case, by the same lemma, we obtain $\klam{\ult_3,\ult_2,\ult_1}\in Q_f\cap-S_g$. In consequence, by the construction of $\uR$ and by the fact that $\klam{\ultVeps_1,\ultVeps_2,\ultVeps_3}$ is not pure we obtain $\klam{\ultVeps_3,\ultVeps_2,\ultVeps_1}$ is not pure either and
    \[
\klam{\ultVeps_3,\ultVeps_2,\ultVeps_1}\in c(\ult_3,\ult_2,\ult_1)\setminus\{\klam{\ult_3,\ult_2,\ult_1},\klam{\ult_3',\ult_2',\ult_1'}\}\subseteq\uR.\qedhere
    \]
\end{proof}
However, as subsequent propositions show, this is as far as we can, in general, get with the properties of $\uR$ shared with the betweenness. 

\pagebreak

\begin{proposition}
    For no free ultrafilter $\ult$, $\uR(\ulteps,\ulteps,\ulteps)$, so in general $\uR$ does not meet \eqref{BT0}.
\end{proposition}
\begin{proof}
    By the construction of $\uR$, the only triples of the form $\klam{\ulteps,\ulteps,\ulteps}$ that are in $\uR$ can only be in the cells of triples in $S_g$. However, by \cite[Theorem 6.10]{dgm23}, if $S_g(\ult,\ult,\ult)$, then $\ult$ must be a principal ultrafilter. Therefore, if $\uR(\ulteps,\ulteps,\ulteps)$, then $\ulteps$ is either a principal ultrafilter or a copy of a principal ultrafilter.
\end{proof}
\begin{proposition}
    In general, $\uR$ does not meet either \eqref{BT2} or \eqref{BT3}.
\end{proposition}
\begin{proof}
  The proof is based on \cite[Example 8.6]{dgm23}. We take the 3-frame $\frN=\klam{\omega,\Bleqs}$ where $B_{\leqslant}$ is the betweenness relation obtained from the standard binary order $\leqslant$ on the set $\omega$ of natural numbers. Let $2^\omega$ be the set of all sets of natural numbers. We consider the complex algebra $\Cm(\frN)=\klam{2^\omega,\QB,\SB}$, its canonical frame $\Cf(\Cm(\frN))=\klam{\Ult 2^\omega,\QB,\SB}$, and its special frame $\klam{\uUltN,\uR}$. 
  
  For any triple of free ultrafilters $\ult_1,\ult_2,\ult_3$ it is the case that $\QB(\ult_1,\ult_2,\ult_3)$. Recall that also $-\SB(\ult_1,\ult_2,\ult_3)$ for any such triple. Therefore $\uR(\ult_1,\ult_2,\ult_3')$, since 
  \[
    c(\ult_1,\ult_2,\ult_3)\setminus\{\klam{\ult_1,\ult_2,\ult_3},\klam{\ult_1',\ult_2',\ult_3}\}\subseteq\uR.
  \]
  But $-\uR(\ult_1,\ult_1,\ult_2)$, as $\QB(\ult_1,\ult_1,\ult_2)$ and $-\SB(\ult_1,\ult_1,\ult_2)$. Con\-se\-quen\-tly, \eqref{BT2} fails for the special frame $\klam{\uUltN,\uR}$.

  \smallskip

  We use the same construction to show the failure of \eqref{BT3}. Again, fix a triple of free ultrafilters $\ult_1,\ult_2,\ult_3$. From the above it follows that $\uR(\ult_1,\ult_2,\ult_3')$ and $\uR(\ult_1,\ult_3',\ult_2)$, but obviously $\ult_2\neq\ult_3'$. 
\end{proof}

In \cite{dgm23} we also considered two more axioms for the betweenness:
\begin{gather}
    B(a,b,a)\rarrow a=b,\tag{BTW}\\
    B(a,a,b).\tag{BT2s}
\end{gather}
The reader will easily check that the special model $\klam{\uUltN,\uR}$ does not satisfy any of them.

\section*{Funding}

\begin{sloppypar}

This research was funded by the National Science Center (Poland), grant number 2020/39/B/HS1/00216, ``Logico-philosophical foundations of geometry and topology''. Ivo D\"untsch  would like to thank Nicolaus Copernicus University in Toru\'n for
supporting him through the ``Excellence Initiative--Research University'' program.

\end{sloppypar}

\enlargethispage{20pt}

\bibliographystyle{plainnat}

\begin{thebibliography}{15}
\providecommand{\natexlab}[1]{#1}
\providecommand{\url}[1]{\texttt{#1}}
\expandafter\ifx\csname urlstyle\endcsname\relax
  \providecommand{\doi}[1]{doi: #1}\else
  \providecommand{\doi}{doi: \begingroup \urlstyle{rm}\Url}\fi

\bibitem[Blackburn et~al.(2001)Blackburn, de~Rijke, and Venema]{brv_modal}
P.~Blackburn, M.~de~Rijke, and Y.~Venema.
\newblock \emph{Modal logic}, volume~53 of \emph{Cambridge Tracts in
  Theoretical Computer Science}.
\newblock Cambridge University Press, 2001.
\newblock ISBN 0-521-80200-8; 0-521-52714-7.
\newblock \doi{10.1017/CBO9781107050884}.
\newblock URL \url{https://doi.org/10.1017/CBO9781107050884}.
\newblock MR1837791.

\bibitem[Burris and Sankappanavar(2012)]{bs_ua}
S.~Burris and H.~P. Sankappanavar.
\newblock \emph{A Course in Universal Algebra - The Millenium Edition, 2012
  Update}.
\newblock Springer Ver\-lag, New York, 2012.
\newblock URL \url{https://www.math.uwaterloo.ca/~snburris/htdocs/ualg.html}.
\newblock MR0648287.

\bibitem[D\"{u}ntsch and Or{\l}owska(1999)]{Duntsch-et-al-MMASO}
I.~D\"{u}ntsch and E.~Or{\l}owska.
\newblock Mixing modal and sufficiency operators.
\newblock \emph{Bull. Sect. Logic Univ. {\L}\'{o}d\'{z}}, 28\penalty0
  (2):\penalty0 99--106, 1999.
\newblock ISSN 0138-0680.
\newblock MR1740612.

\bibitem[D\"{u}ntsch and Or{\l}owska(2001)]{do_mixalg}
I.~D\"{u}ntsch and E.~Or{\l}owska.
\newblock Beyond modalities: {S}ufficiency and mixed algebras.
\newblock In \emph{Relational methods for computer science applications},
  volume~65 of \emph{Stud. Fuzziness Soft Comput.}, pages 263--285. Physica,
  Heidelberg, 2001.
\newblock ISBN 3-7908-1365-6.
\newblock \doi{10.1007/978-3-7908-1828-4\_16}.
\newblock URL \url{https://doi.org/10.1007/978-3-7908-1828-4_16}.
\newblock MR1858531.

\bibitem[D\"{u}ntsch et~al.(2017)D\"{u}ntsch, Or{\l}owska, and
  Tinchev]{dot_mixed}
I.~D\"{u}ntsch, E.~Or{\l}owska, and T.~Tinchev.
\newblock Mixed algebras and their logics.
\newblock \emph{J. Appl. Non-Class. Log.}, 27\penalty0 (3-4):\penalty0
  304--320, 2017.
\newblock ISSN 1166-3081.
\newblock \doi{10.1080/11663081.2018.1442138}.
\newblock URL \url{https://doi.org/10.1080/11663081.2018.1442138}.
\newblock {MR3779213}.

\bibitem[D\"untsch et~al.(2023)D\"untsch, Gruszczy\'nski, and Mench\'on]{dgm23}
I.~D\"untsch, R.~Gruszczy\'nski, and P.~Mench\'on.
\newblock Betweenness algebras.
\newblock \emph{Journal of Symbolic Logic}, pages 1--25, 2023.
\newblock URL \url{https://doi.org/10.1017/jsl.2023.86}.
\newblock Published online by Cambridge University Press, 15 November 2023.

\bibitem[Gargov et~al.(1987)Gargov, Passy, and Tinchev]{gpt87}
G.~Gargov, S.~Passy, and T.~Tinchev.
\newblock Modal environment for {B}oolean speculations (preliminary report).
\newblock In \emph{Mathematical logic and its applications ({D}ruzhba, 1986)},
  pages 253--263. Plenum, New York, 1987.
\newblock ISBN 0-306-42599-8.
\newblock URL \url{https://doi.org/10.1007/978-1-4613-0897-3_17}.
\newblock MR0945200.

\bibitem[Goranko(1988)]{gor88}
V.~Goranko.
\newblock Definability in a modal language with necessity and sufficiency.
\newblock \emph{C. R. Acad. Bulgare Sci.}, 41\penalty0 (4):\penalty0 9--11,
  1988.
\newblock ISSN 0366-8681.
\newblock MR0948060.

\bibitem[Goranko and Passy(1992)]{gp92}
V.~Goranko and S.~Passy.
\newblock Using the universal modality: {G}ains and questions.
\newblock \emph{J. Logic Comput.}, 2\penalty0 (1):\penalty0 5--30, 1992.
\newblock ISSN 0955-792X,1465-363X.
\newblock \doi{10.1093/logcom/2.1.5}.
\newblock URL \url{https://doi.org/10.1093/logcom/2.1.5}.
\newblock MR1166502.

\bibitem[Humberstone(1983)]{Humberstone-IW}
I.~L. Humberstone.
\newblock Inaccessible worlds.
\newblock \emph{Notre Dame J. Formal Logic}, 24\penalty0 (3):\penalty0
  346--352, 1983.
\newblock ISSN 0029-4527,1939-0726.
\newblock URL \url{http://projecteuclid.org/euclid.ndjfl/1093870378}.
\newblock MR0703499.

\bibitem[J\'{o}nsson and Tarski(1951)]{jt51}
B.~J\'{o}nsson and A.~Tarski.
\newblock Boolean algebras with operators. {I}.
\newblock \emph{Amer. J. Math.}, 73:\penalty0 891--939, 1951.
\newblock ISSN 0002-9327.
\newblock \doi{10.2307/2372123}.
\newblock URL \url{https://doi.org/10.2307/2372123}.
\newblock MR44502.

\bibitem[Koppelberg(1989)]{kop89}
S.~Koppelberg.
\newblock \emph{{G}eneral {T}heory of {B}oolean {A}lgebras}, volume~1 of
  \emph{Handbook of {B}oolean algebras, edited by J.D. Monk and R. Bonnet}.
\newblock North-Holland, Amsterdam, 1989.
\newblock ISBN 0-444-70261-X.
\newblock MR0991565.

\bibitem[Vakarelov(1989)]{vak89}
D.~Vakarelov.
\newblock Modal logics for knowledge representation systems.
\newblock In A.~R. Meyer and Mikhail Zalessky, editors, \emph{Logic at Botik,
  Symposium on Logic Foundations of ComputerScience, Pereslavl--Zalessky},
  volume 363 of \emph{Lecture Notes in CoputerScience}, pages 257--277, Berlin,
  Heidelberg, 1989. Springer Verlag.
\newblock ISBN 978-3-540-46180-7.
\newblock URL \url{https://doi.org/10.1007/3-540-51237-3_21}.
\newblock MR1030582.

\bibitem[Van~Benthem(1979)]{vanBenthem-MDL}
J.~Van~Benthem.
\newblock Minimal deontic logics.
\newblock \emph{Polish Acad. Sci. Inst. Philos. Sociol. Bull. Sect. Logic},
  8\penalty0 (1):\penalty0 36--42, 1979.
\newblock ISSN 0138-0680.
\newblock MR0534740.

\bibitem[Venema(2007)]{ven07}
Y.~Venema.
\newblock Algebras and coalgebras.
\newblock In Patrick Blackburn, J.~F. A.~K. van Benthem, and Frank Wolter,
  editors, \emph{Handbook of Modal Logic}, volume~3 of \emph{Studies in logic
  and practical reasoning}, pages 331--426. North-Holland, 2007.
\newblock \doi{10.1016/S1570-2464(07)80009-7}.
\newblock URL \url{https://doi.org/10.1016/s1570-2464(07)80009-7}.
\newblock MR3618508.

\end{thebibliography}

\providecommand{\noop}[1]{}

\end{document}